\documentclass{amsart}

\usepackage[margin=1.5in]{geometry}
\usepackage[utf8]{inputenc}

\usepackage{amsmath,amssymb,amsthm}
\usepackage{amsfonts}
\usepackage{mathtools}
\usepackage{enumerate}
\usepackage{hyperref}

\newcommand{\ti}[1]{\tilde{#1}}
\newcommand{\ul}[1]{\underline{#1}}

\newcommand{\R}{\mathbb{R}}
\newcommand{\C}{\mathbb{C}}
\newcommand{\D}{\mathbb{D}}
\newcommand{\Z}{\mathbb{Z}}
\newcommand{\Q}{\mathbb{Q}}
\newcommand{\po}{\partial}

\newcommand{\al}{\alpha}
\newcommand{\Om}{\Omega}
\newcommand{\vp}{\varphi}
\newcommand{\uta}{{\underline{a}}}

\newcommand{\N}{\mathbb{N}}

\newcommand{\F}{{\mathcal{F}}}
\newcommand{\mi}{\mathrm{i}}
\newcommand{\ve}{\varepsilon}
\newcommand{\te}{\theta}
\newcommand{\LL}{{\mathcal{L}}}
\newcommand{\und}{\underline}
\newcommand{\ute}{\underline{\theta}}
\newcommand{\lv}{\left\lvert}
\newcommand{\rv}{\right\rvert}
\newcommand{\lV}{\left\lVert}
\newcommand{\rV}{\right\rVert}

\theoremstyle{plain}
\newtheorem{theorem}{Theorem}[section]

\theoremstyle{plain}
\newtheorem{corollary}[theorem]{Corollary}

\theoremstyle{definition}
\newtheorem{definition}[theorem]{Definition}

\theoremstyle{remark}
\newtheorem{remark}[theorem]{Remark}

\theoremstyle{plain}
\newtheorem{lemma}[theorem]{Lemma}

\title[persistent tangencies for homoclinic bifurcations of automorphisms of $\C^2$]{Stable intersections of Cantor sets and positive density of persistent tangencies for homoclinic bifurcations of automorphisms of $\C^2$}
\author{Hugo Araújo and Carlos Gustavo Moreira}

\begin{document}
 
\maketitle

\begin{small}
ABSTRACT. Let $\{f_\mu\}_{\mu \in \mathbb{D}}$ be a family of automorphisms of $\C^2$ unfolding a generic homoclinic tangency associated to a fixed point $p$ belonging to a horseshoe. We prove that if the linearized versions of the Cantor sets representing the local intersections of the stable and unstable manifolds of $p$ with the horseshoe have stable intersections, then the set of parameters $\mu$ corresponding to automorphisms with persistent tangencies has positive density at $\mu = 0$.

\end{small}

\section{Introduction}
 In the study of smooth dynamical systems, hyperbolicity is a property related to stability and codification of the dynamics. One manner in which a system fails to be hyperbolic is by having homoclinic tangencies. Although homoclinic tangencies have been found to be behind many sophisticated behaviors, their understanding is far from complete. 
 
 In the holomorphic setting, even less is known about them. The goal of this paper is to improve the understanding of the subject through the study of families unfolding homoclinic tangencies. We are interested in the frequency with which the dynamics is non-hyperbolic in this type of family.
 
One of the pioneers in their study was Newhouse. In \cite{n_1}, considering a particular homoclinic tangency, he proved the existence of open sets of $C^2$ diffeomorphisms of a compact surface having persistent homoclinic tangencies, implying the existence of robustly non-hyperbolic dynamics. In \cite{n_2}, he showed that the coexistence of infinitely many sinks is a generic property in these open sets of persistent tangencies, the so-called Newhouse phenomenon. Later, Newhouse proved that this occurs near any diffeomorphism having a homoclinic tangency (\cite{n_3}). Therefore, non-hyperbolicity and infinitely many sinks were established to be non-negligible behaviors from the topological viewpoint.

A crucial part of his strategy was to associate tangent intersections between the stable and unstable laminations of a horseshoe with intersections between Cantor sets. For this reason, criteria for intersections between Cantor sets have been linked to the study of two-dimensional dynamics since the 70s.

From a measure theoretical point of view, one can study parameterized families of diffeomorphisms. In \cite{palistakens}, Palis and Takens considered families $\{f_\mu\}_{\mu \in (-1,1)}$ of surface diffeomorphisms that ``unfold'' a homoclinic tangency for $f_0$. These families represent a path in the space of diffeomorphisms coming ($\mu < 0$) from the region of hyperbolicity and passing through a bifurcation at the value $\mu=0$. In their work, it was shown that the fractal dimension of the horseshoe to which the homoclinic tangency was associated played a role in the frequency with which the systems $f_\mu$ are hyperbolic for $\mu >0$: if its Hausdorff dimension is less than $1$, then the parameters corresponding to hyperbolicity have full Lebesgue density at $0$.

Some years later, Moreira (\cite{tesegugu}) established that if certain Cantor sets constructed from the dynamics have stable intersections, then the set of parameters $\mu$ for which the diffeomorphism $f_\mu$ has persistent tangencies has positive inferior density at $\mu = 0$. Thus, under this hypothesis, persistent tangencies are not negligible from the measure theoretical point of view as well. This result was later improved by him and Yoccoz in \cite{my2}, where the same conclusion was obtained under the hypothesis that the Hausdorff dimension of the horseshoe is larger than $1$. 

 In the last years, there have been various efforts on extending important results on the dynamics of diffeomorphisms of real manifolds to the holomorphic context. This paper aims at extending results about homoclinic bifurcations to the context of automorphisms of $\C^2$.

An automorphism of $\C^2$ is a holomorphic map $H: \C^2 \to \C^2$ with a holomorphic inverse $H^{-1}: \C^2 \to \C^2$. We denote the space of such automorphisms by $Aut(\C^2)$. As shown in \cite{Lempert1992}, this space has a non-trivial structure. We equip this space with the compact-open topology.

Some results related to homoclinic tangencies and the Newhouse phenomenon in the context of automorphisms of $\C^2$ are already known. In \cite{Buzz}, Buzzard extends the results from Newhouse's works \cite{n_1} and \cite{n_2} to this setting, leaving open whether the associated phenomena occur near any homoclinic tangency. Recently, Avila, Lyubich and Zhang announced that this indeed holds.

The existence of families unfolding generic tangencies inside the complex Hénon family was shown by Fornaess and Gavosto in \cite{forngav}. Gavosto also investigated the presence of the Newhouse phenomenon in these families in her work \cite{gavosto}. Recently, Biebler (\cite{biebler}) has developed a complex version of the criterion used by Newhouse to study intersections of Cantor sets in the real line.

At the same time, the work \cite{DujardinLyu2015} of Dujardin and Lyubich shows that in certain parameterized families $\{f_\lambda\} _{\lambda \in \Lambda}$ ($\Lambda$ is a complex manifold) of polynomial automorphisms of $\C^2$ there is a dense subset of the bifurcation parameters corresponding to homoclinic tangencies. This kind of bifurcation refers to the change of type of periodic orbits: attracting, saddle, repelling, indifferent. It is important to observe that this concept of bifurcation is related to concepts of stability and hyperbolicity, as proved in \cite{berger}.

In this work, we also consider families $\{f_\mu\}_{\mu \in \D}$ of automorphisms unfolding a generic tangency, as done in \cite{forngav} and \cite{gavosto}. The main result of this paper is very inspired by theorem I.2 of \cite{tesegugu}. We show that if some Cantor sets associated to a generic homoclinic tangency have stable intersections, then the set of parameters $\mu \in \D$ for which there are persistent tangencies between the stable and unstable laminations of the continuation of the horseshoe has positive inferior density at $\mu = 0$. This means that, under this hypothesis, persistent tangencies are a relevant behavior close to a homoclinic tangency from a measurable point of view as well.

More precisely, let $X \subset \D$ be a subset of parameters. We say that $X$ has positive inferior density at $\mu = 0$ when
\[ \liminf_{\rho\rightarrow 0} \dfrac{m(X\cap \rho \cdot \mathbb{D})}{m(\rho \cdot \mathbb{D})}>0, \]
where $m$ denotes the Lebesgue measure on the disk. We prove the following theorem.

 \begin{theorem}\label{main}
Let $\{H_\mu\}_{\mu \in \mathbb{D}}$ be a family of automorphisms of $\C^2$ unfolding a generic homoclinic tangency associated to a horseshoe $\Lambda_0$ of $H_0$. Let $H_\mu$ be the continuation of the horseshoe $H_0$ and $W^s(\Lambda_\mu)$ and $W^u(\Lambda_\mu)$ be its stable lamination and its unstable lamination respectively. Also, let $k^{\ul{\theta}^u}({K}_0^u$) and $k^{\ul{\theta}^s}({K}_0^s)$ represent the linearized versions of the unstable and stable Cantor sets associated to the horseshoe of $H_0$.

There exists a $\zeta \in  \C^*$ such that if
$$ (k^{\ul{\theta}^u}({K}_0^u), \, \zeta(k^{\ul{\theta}^s}({K}_0^s)) + \nu )$$
is a pair of configurations having stable intersections for some $\nu \in \C$, then the set of persistent tangencies, 
 \[
    C_{pers} = \{ \mu \in \D,\; H_\mu \text{ has persistent tangencies between } W^s(\Lambda_\mu) \text{ and } W^u(\Lambda_\mu)  \}
    \]

has positive inferior density at $\mu = 0$.

\end{theorem}

For precise definitions of the terms in the theorem stated we refer to reader to section \ref{sec2} and the beginning of section \ref{sec3}. Our analysis of Cantor sets in this context is based on our previous work \cite{araujo_moreira_2023.1}. There we obtained results relating Cantor sets to horseshoes and a criterion for stable intersection different from the one developed by Biebler in \cite{biebler}.

We also would like to highlight that, under generic conditions on the eigenvalues of $DH_0$ at the fixed point associated to the homoclinic tangency, the condition appearing in the statement of Theorem \ref{main} can be changed to a weaker one. See remark \ref{final} for details. 

The text is organized as follows. In section \ref{sec2} we present the results available in the literature that are relevant to our work. Subsection \ref{confcant} is dedicated to results appearing in \cite{araujo_moreira_2023.1} related to the structure of conformal Cantor sets and has pointers to the location of the original statements. In subsection \ref{sec2.1}, we focus on the presentation of some key concepts from hyperbolic dynamics and important properties of stable and unstable foliations near a horseshoe. In section \ref{sec3}, we describe the characteristics of parameterized families unfolding a generic homoclinic tangency and prove some necessary results about their geometry. In section \ref{sec4} we prove Theorem \ref{main} and remark that, under appropriate conditions, the hypothesis of Theorem \ref{main} can be considerably weakened. In the \hyperref[appendix]{Appendix} we address the hypothesis of Theorem \ref{main}, referring to examples and conditions under which the linearized versions of the Cantor sets have stable intersection, and showing that the unfolding of a generic homoclinic tangency occurs arbitrarily close to any homoclinic tangency. It has come to our attention that Dujardin has established a stronger version of the result we prove in this paper, for families of dissipative polynomial automorphisms of constant dynamical degree in which a homoclinic tangency is not persistent. See Theorem A.2 of \cite{dujardin2023}.

 \section{Preliminaries}\label{sec2}

In this section, we cover some definitions and results needed for the statement and proof of Theorem \ref{main}. For further details on them, we refer the reader to \cite{araujo_moreira_2023.1}.

\subsection{On conformal Cantor sets in the complex plane}  \label{confcant} \hfill\\

First, we remember that a $C^m$ \emph{regular Cantor set} on $\C$, also called \emph{dynamically defined Cantor set}, is given by the following data.

\begin{itemize}

\item A finite set $\mathbb{A}$ of letters and a set $B \subset \mathbb{A} \times \mathbb{A} $ of admissible pairs.

\item For each $a\in \mathbb{A}$ a compact connected set $G(a)\subset \mathbb{C}$.
 
\item A $C^{m}$ map $g: V \to \mathbb{C}$, for $m>1$, defined on an open neighborhood $V$ of $\bigsqcup_{a\in \mathbb{A}}G(a)$.
\end{itemize}

Notice that we do not require the map $g$ to be holomorphic, and so it can be considered as a $C^m$ map from an open subset of $\R^2$ to $\R^2$. These data must verify the following assumptions:

\begin{itemize}
    
\item The sets $G(a)$, $a \in \mathbb{A}$, are pairwise disjoint.

\item $(a,b)\in B $ implies $G(b) \subset g(G(a))$, otherwise $G(b) \cap g(G(a)) = \emptyset$.

\item For each $a \in \mathbb{A}$, the restriction $g|_{G(a)}$ can be extended to a $C^{m}$ diffeomorphism from an open neighborhood of $G(a)$ onto its image such that $m(Dg) > \mu$ for some constant $\mu > 1$, where $m(A) := \displaystyle{\inf_{v \neq 0}\frac{\lvert Av \rvert}{\lvert v \rvert}}$ is the minimum norm of the linear operator $A$ on $\mathbb{R}^2$.

\item The subshift $(\Sigma, \sigma)$ induced by $B$, called the type of the Cantor set, 
\begin{align*}
    \Sigma & =\{(a_0, a_1, a_2, \dots  ) \in \mathbb{A}^{\mathbb{N}}:(a_i,a_{i+1}) \in B, \forall i \geq 0\}\\
    \sigma & (a_0,a_1,a_2, \dots)  = (a_1,a_2,a_3, \dots),
\end{align*}
is topologically mixing.

\end{itemize}

Once we have all these data we can define a Cantor set (i.e. a totally disconnected, perfect compact set) on the complex plane: 
\[ K=\bigcap_{n \geq 0} g^{-n}\left(  \bigsqcup_{a \in \mathbb{A}} G(a) \right). \] 

\begin{definition}
    We say that a regular Cantor set $K \subset \C$ is conformal if, for all $x \in K$, the derivative of $g$ at $x$, denoted by $Dg(x) : \R^2 \to \R^2$, is a conformal linear operator. 
\end{definition}

For brevity, we will call these sets ``conformal Cantor sets'' in this paper, instead of conformal dynamically defined (or regular) Cantor sets. Notice that a conformal Cantor set $K$ can be described in multiple ways as a Cantor set constructed from the objects above. In particular, one may enlarge the sets $G(a)$ so that they are the closure of open sets if necessary (this is described is right before Lemma 2.1 of \cite{araujo_moreira_2023.1}).

In this spirit, whenever we refer to a conformal Cantor set $K$ we assume that one particular set of data as above has been fixed. We may as well refer to the Cantor set defined by the map $g$, since all the data can be inferred if we know $g$.

\begin{definition}{(The space $\Om^{m}_{\Sigma}$)} \label{topo} For a fixed symbolic space $\Sigma$ and real number $m>1$, the set of all $C^m$ conformal regular Cantor sets $K$ with the type $\Sigma$ is defined as the set of all $C^m$ conformal Cantor sets described as above whose set of data includes the alphabet $\mathbb{A}$ and the set  $B$ of admissible pairs used in the construction of $\Sigma$. We denote it by $\Om^{m}_{\Sigma}$.

\end{definition} 

The topology in  $\Om^{m}_{\Sigma}$ is generated by a basis of neighborhoods $U_{K,\delta} \subset \Om^{m}_{\Sigma}$ where $ K $ is any $C^{m}$ Cantor set in $ \Om^m_{\Sigma}$ and $ \delta  $ is larger than $0$. The neighborhood $U_{K,\delta}$ is the set of all $C^m$ conformal regular Cantor sets $K'$ given by $g': V' \to \C, \, V' \supset \bigsqcup_{a \in \mathbb{A}} G'(a)$ such that $G(a) \subset V_{\delta}(G'(a))$, $G'(a) \subset V_{\delta}(G(a))$ (that is, the pieces are close in the Hausdorff topology) and the restrictions of $g'$ and $g$ to $V \cap V'$  are $\delta$ close in the $C^{m}$ metric. 

To study the combinatorial and geometrical properties of these sets, we will need to consider the following.
\begin{align*}
    \Sigma^{fin} & = \{(a_0, \dots ,a_n): (a_i,a_{i+1}) \in B \ \forall i ,\, 0 \leq i < n \}, \\
\Sigma^- & = \{(\dots, a_{-n}, a_{-n+1},\dots,a_{-1},a_0): (a_{i-1},a_i) \in B \ \forall i \leq 0\}.
\end{align*} 

For $\uta=(a_0, a_1, \dots , a_n) \in \Sigma^{fin}$, we say that it has length $n$ and define:
\[G(\uta)= \{x \in \bigsqcup_{a \in \mathbb{A}} G(a) , \; g^j(x) \in G(a_j), \;j=0,1,\dots, n \}.
\]
These sets are called \emph{pieces} of the Cantor set $K$. The function $f_{\uta}: G(a_n) \to G(\uta)$ is defined by:

\[ f_{\uta} =  g|^{-1}_{G(a_0)} \circ g|^{-1}_{G(a_1)} \circ \dots \circ (g|^{-1}_{G(a_{n-1})})|_{G(a_n)} . \]

The diameters of the pieces $G(\uta)$ decay exponentially with the length of $\uta$, according to Lemma 2.1 of \cite{araujo_moreira_2023.1}:
\begin{lemma} \label{lemma: decay}
Let $K$ be a conformal Cantor set and $G(\uta)$ the sets defined above. There exists a constant $C>0$ such that:

$$diam(G(\uta)) < C\mu^{-n}.$$

\end{lemma}

Given $\uta=(a_0, \dots, a_n)$, $\und{b}=(b_0, \dots , b_m)$, $\ute^1=(\dots,\theta_{-2}^1,\theta_{-1}^1,\theta_{0}^1)$, $\und{\theta}^2=(\dots,\theta^2_{-2},\theta^2_{-1},\theta^2_{0})$, 
we write:
\begin{itemize}
\item if $a_n=b_0$, $\und{ab}=(a_0, \dots,a_{n}, b_1, \dots, b_m)$;
\item $\und{\theta}^1_n = (\theta^1_{-n}, \dots, \theta^1_{-1},\theta^1_0)$;
\item if $\theta^1_0=a_0$, $\ute \uta = (\dots, \theta^1_{-2},\theta^1_{-1}, a_0, \dots, a_n )$;
\item if $\und{\theta}^1 \neq \und{\theta}^2$ and $\theta ^1 _0 = \theta ^2_ 0$,   $\und{\theta}^1 \wedge \und{\theta}^2=(\theta_{-j}, \theta_{-j+1}, \dots , \theta_0)$, in which $\theta _{-i} =\theta^1 _{-i}= \theta^2 _{-i} $ for all $i=0, \dots, j $ and $\theta^1_{-j-1} \neq \theta^2_{-j-1}$;
\end{itemize}
and define the distance between $\ute^1$ and $\ute^2$ by $d(\ute^1,\ute^2)=diam(G(\ute^1\wedge \ute^2))$. Notice that $f_{\uta}\circ f_{\und{b}} = f_{\und{ab}}$. 

For each $a \in \mathbb{A}$, fix a point $c_a\in K(a) \coloneqq K \cap G(a)$. For any $\ul{\theta} \in \Sigma^-$ and $n \ge 1$, define $\Phi_{\ute_n}$ as the unique map in 
\[
Aff(\C) \coloneqq \{\alpha z + \beta,\ \alpha \in \C^*,\,\beta \in \C\}
\]
such that 
\[
\left(\Phi_{\ute_n} \circ f_{\ute_n}\right) (c_{\te_0}) = 0 \qquad \text{and} \qquad D\left(\Phi_{\ute_n} \circ f_{\ute_n}\right) (c_{\te_0}) = Id.
\]

The existence of $\Phi_{\ute_n}$ is a consequence of the conformality of $g$ over $K$. Consider then the maps
\[
k^{\ute}_n \coloneqq \Phi_{\ute_n} \circ f_{\ute_n}.
\]
They correspond to \emph{normalized} versions of small pieces of the Cantor set $K$. More importantly, when $n \rightarrow +\infty$, these maps converge to maps $ k^{\und{\theta}}: G(\theta_0) \to \C $ called \emph{limit geometries} of $K$, as shown in Lemma 3.1 of \cite{araujo_moreira_2023.1}.

\begin{lemma}\label{lemma:limgeo} (Limit Geometries) For each $\und{\theta} \in \Sigma^-$ the sequence of $C^{m}$ embeddings $k^{\und{\theta}}_n: G(\theta_0) \to \C $ converges in the $C^{m}$ topology to a $C^m$ embedding $ k^{\und{\theta}}: G(\theta_0) \to \C $. The convergence is uniform over all $\und{\theta} \in \Sigma^-$ and in a small neighborhood of $g$ in $\Om^{m}_{\Sigma}$. 
\end{lemma}

\begin{remark}
The limit geometries allow us to control the sets $G(\ute_n)$, which can be thought of as approximations of small parts of the Cantor set $K$. Also by Lemma 3.1 of \cite{araujo_moreira_2023.1}, the limit geometries $k^{\ute}$ depend continuously in $\ute$ and the Cantor set $K$, and the derivative $Dk^{\ute}(x)$ is conformal for all $x\in K(\theta_0)$. Moreover, as pointed out just after Corollary 3.2 of \cite{araujo_moreira_2023.1}, changing the base points $c_a \in K(a)$ only changes the limit geometries by a composition on the left by a bounded affine map.
\end{remark} 

\begin{remark}\label{eq1}
    Remark 3.7 of \cite{araujo_moreira_2023.1} affirms that given $\uta \in \Sigma^{fin}$ and $\ute \in \Sigma^-$ such that $\te_0 = a_0$ then there is $F^{\ute, \uta} \in Aff(\C)$ such that
\begin{equation*}
    k^{\ute}\circ f_{\uta} = F^{\ute, \uta} \circ k^{\ute\uta}.
\end{equation*}
\end{remark}

To describe the transversal geometry of the stable and unstable foliations of a horseshoe near to a homoclinic tangency, we need to consider configurations of Cantor sets, as in Definition 3.1 of \cite{araujo_moreira_2023.1}.

\begin{definition}
 A $C^m$-\emph{configuration} of a piece $G(a)$ of a Cantor set is a $C^m$, $m > 1$, diffeomorphism
\[
h: G(a) \to U \subset \C.\]
The space of all $C^m$ configurations of a piece $G(a)$ is denoted by $\mathcal{P}^m(a)$ and we equip it with the $C^{m}$ topology.

\end{definition}

Observe that a limit geometry is a configuration of a piece. 
The concept of stable intersection is defined below.

\begin{definition}\label{stableint}(Stable Intersections)
     Given a pair of Cantor sets $(K,K')$, we say that a pair of configurations $h:G(a) \to \C$ and $h':G'(a') \to \C$ has \emph{stable intersection} when
     \begin{itemize}
         \item  for any Cantor sets $\tilde{K}$ sufficiently close to $K$ and $\tilde{K'}$ sufficiently close to $K'$;
         \item and for any $\tilde{h}$ sufficiently close to $h$ in $\mathcal{P}^m(a)$ and $\tilde{h'}$ sufficiently close to $h'$ in $\mathcal{P}^m(a')$;
     \end{itemize}
 the intersection between $\tilde{h}(\tilde{K}(a))$ and $\tilde{h}'(\tilde{K}'(a'))$ is not empty.
\end{definition} 

Notice that if $(h,h')$ is a pair of configurations with stable intersection and 
\[A \in Aff(\R^2) \coloneqq \{A(v) = Bv + v_0, \, B \in GL(\R^2), \; v, v_0 \in \R^2\}, \] then $(A \circ h, A \circ h')$ also has stable intersection. This leads us to consider, for any given $h \in \mathcal{P}(a)$, the map $A_h \in Aff(\R^2)$ such that
\[
A_h \circ h (c_{a}) = 0  \qquad \text{ and } \qquad  D(A_h \circ h) (c_{a}) = Id.
\]
In this case, $(A_h \circ h, A_h \circ h')$ still has stable intersection.

Given a finite word $\uta_n$ of size $n$ and a configuration $h: G(a_0) \to \C$ we will write $h_n = h \circ f_{\uta_n}$ in the next lemma. Also under this notation, given $\ute \in \Sigma^-$ with $a_0 = \te_0$, let $\mathfrak{h}_{n}$ be defined by
 \[\mathfrak{h}_{n} \coloneqq A_{h_n} \circ h_n \circ (k^{\ute \uta_n})^{-1}  ,\]
 so that
 \[
A_{h_n} \circ h_n = \mathfrak{h}_n \circ k^{\ute \uta_n}.
 \]
 Notice that by remark \ref{eq1}.
 \[
 \mathfrak{h}_n = A_{h_n} \circ h \circ (k^{\ute})^{-1} \circ F^{\ute, \uta_n}.
 \]

We have the following version of Lemma 3.8 of \cite{araujo_moreira_2023.1}.
\begin{lemma}\label{scale}
Let $K$ be a $C^{1+\alpha}$, $0 < \alpha < 1$, conformal Cantor set and $h \in \mathcal{P}^{1+\al}(a_0)$ a configuration of a piece in $K$. Then, using the notation described above,
\[
||\mathfrak{h}_n -Id  ||_{C^{1+\alpha}} < C \cdot \text{diam}(G(\uta_n))^{\alpha}< C \cdot \mu^{-n\alpha},
\] 
where $C> 0$ depends only on the Cantor set $K$ and the initial configuration $h$.
\end{lemma}

Therefore, the configurations $A_{h_n}\circ h_n$ are close to limit geometries ending in $\uta_n$. In other words, under the action of composition with inverse branches of $g$, configurations are attracted  (exponentially) to compositions of limit geometries and affine maps. This lemma will be very important in the proof of the Theorem \ref{main}, as it will help us find stable intersection between configurations of certain Cantor sets. 

\subsection{Some properties of stable and unstable foliations of horseshoes in $\C^2$ and their relations to conformal Cantor sets.} \label{sec2.1} \hfill\\

The concept of conformal Cantor set was developed having in mind the transversal geometries of the stable and unstable foliations of a horseshoe near a periodic point in the context of automorphisms of $\C^2$. Next, we explain the meaning of these terms. 

Given an automorphism $H \in Aut(\C^2)$ we say that a compact set $\Lambda \subset \C^2$ is hyperbolic when it is hyperbolic in the usual sense: it is invariant ($H(\Lambda) = \Lambda$) and there is an invariant splitting of the tangent bundle (over $\Lambda$)  $T\C^2|\Lambda =E^s\oplus E^u$ such that for some $\lambda>0$ and $C>0$ 
\[\max\{||DH^j|E^s||, ||DH^{-j}|E^u\} \leq C\lambda^j\;\; \forall j \in \N.\]

The stable and unstable manifolds ($W^s(p)$ and $W^u(p)$ respectively) of points $p \in K$ are, in this case, complex manifolds of dimension one immersed on $\C^2$ and their local versions are complex embedded disks. 

A basic set is a locally maximal hyperbolic set, meaning there is an open set $U \supset \Lambda$ such that ${\Lambda = \bigcap_{n \in \mathbb{Z}}H^n(U)}$, also having a transitive orbit. Basic sets have the key property of being stable under small perturbations, that is, if $H'$ is close to $H$ then ${\Lambda'= \bigcap_{n \in \mathbb{Z}}{(H')}^n(U)} $ is also a basic set with $H'|\Lambda'$ being a dynamical system conjugated, via a homeomorphism close to identity, to $H|\Lambda$. We say that a basic set is of saddle type when the bundles $E^s$ and $E^u$ are non-trivial. 

A complex horseshoe is a totally disconnected infinite basic set of saddle type of an automorphism of $\C^2$. This nomenclature has already appeared in the literature and is coherent with the one used by Obsterte-Worth in \cite{ob}. In this work, it was shown that complex horseshoes occur when there is a transversal homoclinic intersection. A homoclinic intersection is a point $q \in W^s(p) \cap W^u(p)$ with $q \neq p$, where $p$ is a periodic hyperbolic point of saddle type. We say that the homoclinic intersection is 
\begin{itemize}
    \item transversal when
\[T_q\C^2 = T_qW^s(p) \oplus T_qW^u(p);\]
\item a tangency when 
\[
T_qW^s(p) = T_qW^u(p).
\]
\end{itemize}
Since the stable and unstable spaces over $q$ are complex lines in $\C^2$, a homoclinic intersection is either transversal or a tangency. The following theorem was present in \cite{araujo_moreira_2023.1} (Theorem 2.5), and establishes a relation between the concept of conformal Cantor set presented in the previous subsection and the structure of the horseshoe near a periodic point.

\begin{theorem}\label{thma}
Let $\Lambda$ be a complex horseshoe for a automorphism $H \in Aut(\C^2)$ and $p$ be a periodic point in $\Lambda$. Then, if $\ve$ is sufficiently small, there are an open set $U \subset \C $, an open set $V \subset W^u_\ve(p)$ containing $p$ and a holomorphic parameterization  $\pi: U \to V$ such that $\pi^{-1}(V \cap \Lambda)$ is a $C^{1+\alpha}$ conformal Cantor set in the complex plane.
\end{theorem}

This theorem is also true when we replace $W^u_\ve(p)$ with $W^s_\ve(p)$. Translating if necessary, we can assume that the parameterizations $\pi^u: U^u \to V^u \subset W^u_\ve(p)$ and $\pi^s: U^s \to V^s \subset W^s_\ve(p)$ are such that $\pi^u(0) = p =\pi^s(0)$. We will write
\[
K^u \coloneqq (\pi^u)^{-1} \left( V^u \cap \Lambda\right) \qquad \text{and} \qquad K^s \coloneqq (\pi^s)^{-1} \left( V^s \cap \Lambda\right) 
\]
for the corresponding Cantor sets.

If $p$ is a periodic point of $H$, then it is a fixed point for some iterate of $H$. So, to simplify things, we are going to consider from now on that $p$ is actually a fixed point of $H$ (except in the statement of  lemma \ref{projections}, which was kept close to the original).

The maps $g^u$ and $g^s$ defining these Cantor sets are constructed respecting the dynamics of the automorphism $H$. This means that near $0$,
\[
g^u = (\pi^u)^{-1} \circ  H \circ \pi^u \qquad \text{and} \qquad g^s = (\pi^s)^{-1} \circ H^{-1}\circ \pi^s.
\]
Check the proof of Theorem 2.5 of \cite{araujo_moreira_2023.1} for details. Consequently, $0$ is a fixed point of $g^u$ and $g^s$.

\begin{remark}\label{eigenv}
    Moreover, if  $G(a^u)$ is the piece containing $0$ in $K^u$, $G(a^s)$ is the piece containing $0$ in $K^s$, and if $\lambda^u$, $\lambda^s$ are the eigenvalues of $DH(p)$, then
    \[
    f_{(a^u,a^u)}(0) = 0 \qquad \text{and} \qquad Df_{(a^u,a^u)}(0) = (\lambda^u)^{-1};
    \]
    \[
    f_{(a^s,a^s)}(0) = 0 \qquad \text{and} \qquad Df_{(a^s,a^s)}(0) = \lambda^s.
    \]
\end{remark}


\begin{remark}\label{contdepen}
    The Cantor sets $K^u$ and $K^s$ depend continuously on the automorphism $H$. This is also a consequence of their construction. 
\end{remark}
 
The linearized versions of the Cantor sets appearing in Theorem \ref{main} can now be defined:

\begin{definition}\label{linv}
    Let $K^u$ and $K^s$ be the Cantor sets associated to the fixed point $p$ of a horseshoe $\Lambda$ of an map $H$. Let $G(a^u)$ be the piece containing $0$ in $K^u$ and $G(a^s)$ be the piece containing $0$ in $K^s$. If we fix $c_{a^u} = 0$ and $c_{a^s} = 0$, then $k^{\ute^u}({K}^u)$ and $k^{\ute^s}({K}^s))$ are the linearized versions of these Cantor sets, where
    \[
    \ute^u = (\dots, a^u, a^u , a^u) \qquad \text{and} \qquad \ute^s = (\dots, a^s, a^s,a^s).
    \]
\end{definition}
    
\begin{definition}\label{persset}
    Given a horseshoe $\Lambda$, the stable and unstable laminations of $\Lambda$ are, respectively, the unions of the stable manifolds and unstable manifolds of the points in $\Lambda$:
\[
W^s(\Lambda) = \bigcup_{x \in \Lambda} W^s(x) \qquad \text{and} \qquad   W^u(\Lambda) = \bigcup_{x \in \Lambda} W^u(x).
\]
We say that an automorphism $H$ of $\C^2$ has persistent tangencies when there are tangent intersections between the laminations $W^s(\Lambda_{\ti{H}})$ and $W^u(\Lambda_{\ti{H}})$ for every $\ti{H} \in Aut(\C^2)$ sufficiently close to $H$.
\end{definition}

 The set $C_{pers}$ in the statement of Theorem \ref{main} corresponds to parameters $\mu$ for which $H_\mu$ has persistent tangencies. Hence, by definition \ref{persset}, the set $C_{pers}$ is open. As mentioned in the introduction, persistent tangencies imply that the dynamics is robustly non-hyperbolic. 

To prove Theorem \ref{thma}, it is necessary to extend these invariant laminations to semi-invariant foliations covering an open set around the horseshoe $\Lambda$. Since the properties of these foliations are also important in this paper, in particular to understand the unfolding of a generic homoclinic tangency, we present them in some detail below.

First of all, their existence is guaranteed by the following theorem, which is a slight variation of a result by Pixton present in \cite{pix} (refinement 3.9 of Theorem 3.4) and almost the same as the version appearing in Buzzard's work \cite{Buzz} (Theorem 5.3).

\begin{theorem} \label{fol}Let $U \subset \mathbb{C}^2$ be open. Let $\Lambda \subseteq U $ be a horseshoe for an injective holomorphism $H_0: U \to M$, with $\Lambda=\displaystyle{\bigcap_{n\in \mathbb{Z}}H_0^n(U)}$ and let $E^s \oplus E^u $ be the associated splitting of $T_{\Lambda}\mathbb{C}^2$.

Suppose that for some $\alpha > 0$ sufficiently small:
\[
\lv DH_0|_{E^s(x)} \rv \cdot  \lv DH_0^{-1}|_{E^u(x)} \rv \cdot \lv  DH_0^{-1}|_{E^s(H_0(x))} \rv^{1+\alpha}< C < 1 
\]
for every $x \in \Lambda$, where $C < 1$ is a constant uniform for all $x$. 

Then, there is a compact set $L$ and $\delta>0$ such that for any holomorphism $H : U \to \mathbb{C}^2$ with $||H-H_0||<\delta$ we can construct a $C^{1+\alpha}$ foliation $\F_H ^u $ defined on an open set $V\subset U$ such that:

\begin{itemize}

\item the horseshoe $\Lambda_H = \displaystyle{\bigcap_{n \in \mathbb{Z}}H^n(U)}$ satisfies $\Lambda_H \subset \text{int } L \subset L \subset \F_H ^u $,
\item if $p \in \Lambda_H $, then the leaf $\mathcal{L}^u_H(p)$ agrees with $W^u_{\text{loc}}(p)$,
\item if $p\in H^{-1}(L)\cap L$, then $H(\mathcal{L}^u_H (p) ) \supseteq \mathcal{L}^u_H (H(p))$, i.e., it is semi-invariant,
\item the tangent space $T_p\mathcal{L}^u_H(p)$ varies $C^{1+\alpha}$ with $p$ and continuously with $H$,
\item The association $H \to \F^u_H$ is continuous on the $C^{1+\alpha}$ topology.

\end{itemize}
\end{theorem}

Of course, by changing $H$ to $H^{-1}$, we guarantee the existence and analogous properties for the stable foliation $\F^s_H$.

\begin{remark}\label{tangsmooth}
    The result above is Theorem 2.2 combined with Remark 2.3 of \cite{araujo_moreira_2023.1}. The main difference between Theorem \ref{fol} and Theorem 5.3 of \cite{Buzz} is the $C^{1+\alpha}$ regularity of the foliations, which is needed for our objectives. The tool used for this purpose was a modified version of the $C^r$-section Theorem as stated in \cite{global}. The $C^{1+\al}$ invariant section obtained in the proof of Theorem 2.2 in the appendix of \cite{araujo_moreira_2023.1} is $(x,\sigma_H(x))$, where $x$ represents points in the open set $U \subset \C^2$ and $\sigma_H(x)$ represents the tangent space of $\LL^u_H(x)$ at the point $x$. This way, not only the foliations are $C^{1+\al}$, but the associations $x \mapsto I^u_H(x)$, where $I^u_H(x)$ represents the inclination of $T_x \LL^u_H(x)$ with respect to some fixed direction, are as well. Moreover, the association $H \mapsto I^u_H$ is continuous in the $C^{1+\al}$ topology.
\end{remark}

\begin{remark}\label{holocur}
    We observe these foliations are not unique; the leaves not corresponding to parts of $W^{s}(p)$ or $W^{u}(p)$ for $p \in \Lambda$ can be chosen in different ways. For further details on their construction, we refer the reader to \cite{pix} and \cite{Buzzthesis}. In this last reference there is also a proof that although the leaves $\mathcal{L}^{s}(x)$ and $\mathcal{L}^{u}(x)$ ($x \in V$) vary only with $C^{1+\alpha}$ regularity, the foliations can be constructed in a way such that each of these leaves are holomorphic curves embedded into $\C^2$. This fact will be crucial in the next section.
\end{remark}

The holonomies along $\F^u_H$ (and $\F^s_H$) between two tranversal sections to it have a very nice property, which is described by Lemma 2.6 of \cite{araujo_moreira_2023.1}:

\begin{lemma}\label{projections}
Let $\Lambda_H$ be a complex horseshoe for an automorphism $H \in Aut(\C^2)$ together with its unstable foliation $\F^u_H$. Additionally, let $N_1$ and $N_2$ be two $C^{1+\alpha}$ transversal sections to $\F^u_H$. Suppose that for some periodic point $p \in \Lambda_H$, the tangent planes of $N_1$ and $N_2$ to the points of intersection $N_1 \cap \mathcal{L}^u_H(p) =q_1$ and respectively $N_2 \cap \mathcal{L}^u_H(p)=q_2$ are complex lines of $\C^2$. Then the projection along unstable leaves $\Pi^u : N_1 \to N_2$ is a $C^{1+\alpha}$ map conformal at $q_1$.
\end{lemma} 

The proof of this result is based on iterating backward the leaf $\mathcal{L}^u(p)$, until the iterates of $N_1$ and $N_2$ are very close to each other. By the inclination lemma, the holonomy between these iterates is very close to being conformal. Because of the semi-invariant properties of the foliation and the holomorphicity of $H$, we can iterate back and forth by $H^n$ to conclude that the original holonomy was also very close to being conformal. Making $n\rightarrow \infty$ yields the result. Again, an analogous result is true for the holonomies along the stable foliation $\F^s_H$.

It has been pointed out to us that Lyubich and Peters use a another argument to prove that these holonomies are in fact $C^{1+\alpha}$, using holomorphic motions, Koebe Distortion Theorem and properties of (quasi)conformal maps. See Lemma 5.3 of \cite{lyubichpeters} for further details. Part of the argument also involves iterating back and forth by the automorphism, as used in the proof of Lemma \ref{projections}.

 \section{The unfolding of a generic homoclinic tangency}\label{sec3}

In this section, we describe what properties characterize a family of automorphisms unfolding a generic homoclinic tangency. We keep them as close as possible to the ones appearing in \cite{forngav} and \cite{gavosto}.

First, by a family of automorphisms $\{H_\mu\}_{\mu \in \D}$ we mean:

\begin{enumerate}[(I)]

    \item \label{itm:I} a $C^r$, $r > 1$, map $H$,
\begin{align*}
    H:\D\times \C^2  & \to \D\times \C^2\\
                (\mu, (z,w))    &  \mapsto  (\mu,H_{\mu}(z,w))  
    \end{align*} 
    such that each $H_\mu$ is an automorphism of $\C^2$.
\end{enumerate}

We are going to assume that the homoclinic tangency is associated to a horseshoe. This is not an unnatural hypothesis, since horseshoes occur when there is a transversal homoclinic tangency (as proved in \cite{ob}), and these can be made by perturbing a homoclinic tangency. By iterating $H$, we may assume that the homoclinic tangency comes from a fixed point, instead of a period one.

\begin{enumerate}[(I)]
\setcounter{enumi}{1}
    \item \label{itm:II} $H_0$ has a fixed saddle point $p=p_0$ belonging to a horseshoe $\Lambda_0$. Also, there is a homoclinic tangency $q$ associated to the saddle point $p$, that is, 
    \[
q \in W^s(p) \cap W^u(p)  \qquad \text{and } \qquad T_qW^s(p) = T_qW^u(p).
    \]
\end{enumerate}

By rescaling the parameter $\mu$ we can assume that all the hyperbolic structures can be continued along $H_\mu $. Thus, we denote by:
\begin{itemize}\setlength\itemsep{.3em}
    \item $p_\mu$ the continuation of the fixed saddle point $p_0$;
    \item $\lambda_\mu^s$ and $\lambda^u_\mu$ are the eigenvalues of $DH_\mu$ at $p_\mu$;
    \item $\Lambda_\mu$ the continuation of the horseshoe $\Lambda_0$;
    \item $W^s_\mu(p)$ and $W^u_\mu(p)$ the continuations of the stable and unstable manifolds;
    \item $K^u_\mu$ and $K^s_\mu$ the continuations of the Cantor sets constructed with Theorem \ref{thma};
    \item $W^s(\Lambda_\mu)$ and $W^u(\Lambda_\mu)$ the stable and unstable laminations of $\Lambda_\mu$;
    \item $\F^s_\mu$ and $\F^u_\mu$ the continuations of the stable and unstable foliations constructed by Theorem \ref{fol}.
\end{itemize}

The points $p_\mu$ depend $C^r$ on $\mu$, and so do the continuations of the unstable and stable manifolds and the foliations. The eigenvalues $\lambda_\mu^u$ and $\lambda^s_\mu$ depend $C^{r-1}$ on $\mu$. The horseshoe and the Cantor sets depend continuously on $\mu$.

For any $r > 0$ and $z_0 \in \C$, let
\[
B_r(z_0)\coloneqq \{z \in \C, \lv z - z_0\rv < r\}.
\]
By iteration, we can extend the foliations $\F^s_\mu$ and $\F^u_\mu$ to a neighborhood of $q$. Maybe after performing a change of coordinates in $GL(\C^2)$, we can assume that in a neighborhood $V$ of the point $q$ the foliations $\F^{u}_\mu$ are described by a $C^{1+\alpha}$ map
\begin{align*}
    \vp^{u} : \D \times B_\delta(0) \times B_{\delta}(0)& \to V \\
                   (\mu,z,w) & \mapsto \vp^{u}(\mu,z,w) = \vp^{u}_{\mu}(z,w).
\end{align*}
such that:
\begin{enumerate}[i)]
    \item for each fixed $\mu \in \D$, the map ${\vp}^u_\mu = {\vp}^u(\mu, \boldsymbol{\cdot})$ is a $C^{1+\alpha}$ diffeomorphism map from $ B_\delta(0) \times B_{\delta}(0)$ to an open subset of $V$ that contains the small open set $Q=B_{\delta'} (q_1) \times B_{\delta'} (q_2)$, where $q=(q_1,q_2)$.
    \item For fixed $\mu$ and $w$, let $\vp^{u}_{\mu,w}(z) \coloneqq \Pi_2 \circ  \vp_{\mu}(z,w)$, where $\Pi_2: \C^2 \to \C$ is the projection onto the second coordinate. For every $\mu \in \D$, 
    \[
    \vp_{\mu}(z,w) = (q_1 + z, \vp^{u}_{\mu,w}(z)).
    \]
    \item For fixed $\mu$ and $w$, the image of $\vp^u_{\mu}(\cdot, w) $ is contained in a leaf $\mathcal{L}^u _\mu$ of the foliation $\F^u_\mu$.
     \item For fixed $\mu$ and $w = 0$, the image of $\vp^u_{\mu}(\cdot, 0) $ is contained in $W^u_\mu(p)$.
    \item For $\mu = 0$, we have
    \[
    \vp^{u}_{0,w}(0) = q_2 + w,
    \]
    and 
    \[
    \frac{\partial \vp^{u}_{0,0}}{\partial z} (0) = 0.
    \]

\end{enumerate}

The maps $\vp^s$, $\vp^s_\mu$ and $\vp^s_{\mu,w}$ are defined analogously. Since $q$ is a point of tangency, we have
\begin{equation}\label{tangencyvp}
    \frac{\partial \vp^{u}_{0,0}}{\partial z} (0) = \frac{\partial \vp^{s}_{0,0}}{\partial z } (0) = 0.
\end{equation}

Moreover, from the conditions above, 
\[
D\vp_0^u(0,0) = Id = D\vp_0^s(0,0).
\]

The homoclinic tangency is said to be generic when

\begin{enumerate}[(I)]
\setcounter{enumi}{2}
    \item \label{itm:III} the contact between $W^u_0(p)$ and $W^s_0(p)$ is quadratic. In other words,
    \[
    \frac{\partial^2 \vp^{u}_{0,0}}{\partial z ^2} (0) \neq \frac{\partial^2 \vp^{s}_{0,0}}{\partial z ^2} (0).
    \]
\end{enumerate}
Notice that these derivatives exist because the leaves are holomorphic (Remark \ref{holocur}).

The final hypothesis on the family is that the stable and unstable manifolds associated to the homoclinic tangency move in a generic fashion according to $\mu$. Consider the function 
\[
c(\mu)=\vp^s_{\mu,0}(0)-\vp^u_{\mu,0}(0).
\]
It can be seen as a real $C^r$ function from $\D\subset \R^2$ to $\R^2$.

\begin{enumerate}[(I)]
\setcounter{enumi}{3}
    \item \label{itm:IV} The generic homoclinic at $q$ is unfolded with ``non-singular speed'':
    \[
    \frac{\partial c} {\partial \mu}(0) = B \in GL(\R^2). 
    \]
\end{enumerate}

\begin{remark}
    The conditions (\ref{itm:I}) - (\ref{itm:III}) are usually considered what determines a generic homoclinic tangency. Meanwhile, condition (\ref{itm:IV}) corresponds to the unfolding of the generic tangency. Notice that all those conditions make reference only to the continuations of $W^u(p)$ and $W^s(p)$ and hence are independent of the non-unique choice on the foliations $\F^u$ and $\F^s$.
\end{remark}

\begin{remark}
    In the works \cite{forngav} and \cite{gavosto}, families satisfying conditions very similar to (\ref{itm:I}) - (\ref{itm:IV}) were denominated ``a generic homoclinic tangency'', without mention to the verb unfold. The work \cite{forngav} shows that a family satisfying conditions (\ref{itm:I}) - (\ref{itm:IV}) exists, in particular, can be chosen within the Hénon family. In \cite{gavosto}, Gavosto proves that, after some change in coordinates, the difference in height between $W^u_\mu(p_\mu)$ and $W^s_\mu(p_\mu)$ near $q$ is equal to the function $z^2 + \mu$, and so the homoclinic tangency is unfolded into two transversal intersections (see proof of Theorem 4.1 there).
    
    Buzzard's work \cite{Buzz} also proves the existence of families unfolding a generic homoclinic tangency through more general constructions. This is not stated directly, but follows from the arguments used in the proof of Corollary 1. In the \hyperref[appendix]{Appendix}, we show that families satisfying conditions (\ref{itm:I}) - (\ref{itm:IV}) can be constructed arbitrarily close to any automorphism $H$ having a homoclinic tangency.
    
\end{remark}

To study tangent intersections between the unstable and stable laminations of $\Lambda_\mu$, we need to fix the part of the neighborhood of the homoclinic tangency $q$ in which they can occur. This is done in the next lemma, which should be familiar to those interested in homoclinic bifurcations.

\begin{lemma}(Disk of tangencies)\label{disktan}
Let $\{H_{\mu}\}_{\mu \in \D}$ be a family of automorphisms of $\C^2$ unfolding a generic homoclinic tangency as described above. 

For every small $\mu \in \D$, there is a $C^{1+\alpha}$ embedded disk $D_\mu \subset V $ such that the leaves $\mathcal{L}^u_\mu(x)$ and $\mathcal{L}^s_\mu(x)$ are tangent to each other at $x$ for every $x \in D_\mu$. Further, the embeddings depend $C^r$ on $\mu$.
\end{lemma}

\begin{proof}
 Consider the map representing the difference between the inclinations of the tangent spaces to the unstable and stable foliations at the point $ x = (z,w) $
 \begin{align*}
  f_\mu(z,w) &  = \left(\frac{\partial  {\vp}^u_{\mu,w}}{\partial z} \circ({\vp}^u_\mu)^{-1} (z,w) \right)
  -  \left( \frac{\partial {\vp}^s_{\mu,w}}{\partial z} \circ ({\vp}^s_\mu)^{-1}\right) (z,w)
  \\
      & = \left( \frac{\partial (\Pi_2\circ {\vp}^u_\mu)}{\partial z} \circ({\vp}^u_\mu)^{-1} \right)(z,w) -\left(\frac{\partial (\Pi_2\circ {\vp}^s_\mu) }{\partial z} \circ ({\vp}^s_\mu)^{-1}\right)(z,w),
 \end{align*}

 where $\Pi_2: \C^2 \to \C$ is the projection onto the second coordinate. By remark \ref{holocur}, the leaves are holomorphic curves and so the functions $\vp^u_{\mu,w}$ and $\vp^s_{\mu,w}$ are holomorphic on $z$. Hence, the derivatives in the definition of $f$ can be seen as $2 \times 2$ conformal matrices or complex numbers. Moreover, by remark \ref{tangsmooth} and identifying $\C$ with $\R^2$, we can consider that each $f_\mu$ is a $C^{1 + \alpha}$ map from an open subset of $\R^4$ to $\R^2$. 
 
 The pre-image $D_0 \coloneqq f_0^{-1}(0)$, correspond to the points near $q$ in which the foliations $\F^s_0$ and $\F^u_0$ are tangent to each other. By hypothesis on $H_0$, the point $q$ belongs to $D_0$. We now show that $\frac{\partial f_0}{\partial z}(q_1,q_2)$ corresponds to a conformal invertible $2 \times 2$ real matrix and so, by the implicit function theorem for real maps, the set $D_0$ is the graph of a $C^{1 + \alpha}$ function of $w$ near $(q_1,q_2)$. Restricting to a small ball around $q$, we may consider that it is a disk.
 
Since
\[
 \vp^{u}_{\mu}(z,w) = (q_1 + z, \vp^{u}_{\mu,w}(z)) \qquad \text{and} \qquad  \vp^{s}_{\mu}(z,w) = (q_1 + z, \vp^{s}_{\mu,w}(z)),
\]
 if we treat $\vp_\mu^u$ and $\vp_\mu^s$ as maps from an open subset of $\R^4$ to $\R^4$, we can represent their derivatives as
 \[
 D\vp^{u}_{\mu}(z,w)  = \begin{pmatrix}
     Id_2 & 0 \\ \frac{\partial }{\partial z} \vp^u_{\mu,w} & \frac{\partial }{\partial w} \vp^u_{\mu,w}
 \end{pmatrix} \qquad \text{and} \qquad  D\vp^{s}_{\mu}(z,w)  = \begin{pmatrix}
     Id_2 & 0 \\ \frac{\partial }{\partial z} \vp^s_{\mu,w} & \frac{\partial }{\partial w} \vp^s_{\mu,w}
 \end{pmatrix} 
 \]
 
 in which each of the entries are $2 \times 2$ real matrices. As $D\vp^u_0(0,0)$ and $D\vp^s_0(0,0)$ are equal to $Id_4$, the matrices above are also close to $ Id_4$ for $(\mu,z,w)$ close to $(0,0,0)$. Using the chain rule, we arrive at:
 \begin{align*}
      \frac{\partial f_0}{\partial z}  (q_1,q_2)  = & \left(\frac{\partial^2(\Pi_2\circ {\vp} ^u_0)}{\partial z^2}(0,0) - \frac{\partial^2 (\Pi_2\circ {\vp} ^s_0)}{\partial z^2}(0,0)  \right)
      \\ 
      & + \frac{\partial}{\partial w}\left(\frac{\partial (\Pi_2\circ {\vp} ^s_0)}{\partial z }\right)(0,0) \cdot \left[\frac{\partial{\vp}^s_{0,0}}{\partial w}(0)\right]^{-1}\cdot \frac{\partial{\vp}^s_{0,0}}{\partial s}(0)
      \\
     & - \frac{\partial}{\partial w}\left(\frac{\partial (\Pi_2\circ {\vp} ^u_0)}{\partial z }\right)(0,0) \cdot \left[\frac{\partial{\vp}^u_{0,0}}{\partial w}(0)\right]^{-1}\cdot \frac{\partial{\vp}^u_{0,0}}{\partial s}(0).
 \end{align*}
 
Notice that these derivatives exist because, by remark \ref{tangsmooth}, the tangent spaces to $\F^u(p) $ and $ \F^s(p)$ vary $C^{1+\al}$ with $p$. Furthermore, each term can be seen as a $2 \times 2$ matrix, again because of remark \ref{holocur}. 

The last two lines of the equation above are equal to zero because of equation (\ref{tangencyvp}). The first line, by condition (\ref{itm:III}) on the family (quadratic tangency) and the fact that each of the leaves is a holomorphic curve, corresponds to the difference between two distinct complex numbers and so is a conformal $2 \times 2$ matrix. We conclude that $\frac{\partial f_0}{\partial z}$ is in fact invertible and so $D_0$ is a $C^{1 + \al}$ embedded disk passing through $q$.
 
 To prove the existence of the disks $D_\mu$ and that they vary $C^r$ with $\mu$, we consider now the map $f: \D \times B_{\delta'}(q_1) \times B_{\delta'}(q_2) \to \C \equiv \R^2$ defined by $f(\mu,z,w) =  f_{\mu}(z,w)$.

Fixing $\mu = 0$, $w = q_2$ and varying only $z$, we conclude that $\frac{\partial f}{\partial z}(0,q_1,q_2) = \frac{\partial f_0}{\partial z}  (q_1,q_2) $ which is an invertible matrix. So the pre-image of $0$ is a $C^r$ manifold, which, close to $(0,0,0)$, is the graph of a $C^r $ function on $\mu$ and $w$. Therefore, for every $\mu$ sufficiently small, the disks $D_\mu$ exist and they vary $C^r$ with $\mu$. To show that each of them is $C^{1+\alpha}$ one needs only to consider the functions $f_\mu(z,w) \coloneqq f(\mu,z,w)$ and observe that, by continuity, one still has $\frac{\partial f_\mu}{\partial z}$ is invertible near $(0,0)$ for small values of $\mu$ and proceed as in the case $\mu = 0$.
\end{proof}

\begin{remark}
   Notice that if the leaves were not constructed as holomorphic curves, we would need to consider $f$ as a map from an open subset of $\R^4$ to $\R^8$, and we would not be able to use the implicit function theorem. 
\end{remark}

Now, following the notation of Theorem \ref{thma}, we can consider local parameterizations
\[
\pi^u_\mu: U^u_\mu \to V^u_\mu \subset W^u_{\mu,\ve}(p_\mu) \qquad \text{ and } \qquad \pi^s_\mu: U^s_\mu \to V^s_\mu \subset W^s_{\mu,\ve}(p_\mu) 
\]
defining the Cantor sets 
\[
K^u_\mu = (\pi^u_\mu)^{-1} \left( V^u_\mu \cap \Lambda_\mu\right) \qquad \text{and} \qquad K^s_\mu = (\pi^s_\mu)^{-1} \left( V^s_\mu \cap \Lambda_\mu\right).
\]
Following the notation of lemma \ref{projections}, let 
\[
\Pi^u_\mu: V^u_\mu \to D_\mu \qquad \text{and} \qquad \Pi^s_\mu: V^s_\mu \to D_\mu
\]
be the projections along $\F^s_\mu$ and  $\F^u_\mu$ respectively. Finally, let $\pi^d_\mu: D_\mu \to \C$ be a parameterization of the disk of tangencies. The following lemma is now evident.

\begin{lemma}\label{evident} For small values of $\mu$, let
   \[
h^u_\mu \coloneqq \pi^d_\mu \circ \Pi^s_\mu \circ \pi^u_\mu \qquad \text{and} \qquad h^s_\mu \coloneqq \pi^d_\mu\circ \Pi^u_\mu \circ \pi^s_\mu .
\] 

If $h^u_\mu(K^u_\mu) \cap h^s_\mu(K^s_\mu) \neq \emptyset$, then there are tangencies between $W^u(\Lambda_\mu)$ and $W^s(\Lambda_\mu)$. 

Additionally, if $h^u_\mu(K^u_\mu)$ and $ h^s_\mu(K^s_\mu)$ have stable intersections, then $\mu$ belongs to the set $C_{pers}$ appearing in the statement of Theorem \ref{main}.
\end{lemma}

To guarantee stable intersections between the configurations of Cantor sets $h^u_\mu(K^u_\mu)$ and $ h^s_\mu(K^s_\mu)$, we need to further develop our understanding of the foliations $\F_\mu^u$ and $\F_\mu^s$. For this purpose, we will need holomorphic linearizations.

\begin{definition}(Holomorphic linearization) Let $H$ be an automorphism of $\C^2$ and $p$ be a fixed saddle point of $H$. Let $\lambda^u$ and $\lambda^s$ be the eigenvalues of $DH(p)$. A holomorphic change of coordinates $\Psi$ near $p$ such that $\Psi(0,0) = p$ and
\[
\Psi^{-1} \circ H \circ \Psi (v^s,v^u) = (\lambda^s\cdot  v^s, \lambda^u \cdot v^u), 
\]
is called a holomorphic linearization of the dynamics near the fixed point $p$.
\end{definition}

In \cite{Bracci2004}, one can find the following theorem, due to Siegel, establishing conditions for the existence of such linearization.
 
\begin{theorem}(Siegel) Let $F$ be a germ of holomorphic diffeomorphism of $\C^n$ fixing $0$ and denote by $\lambda_1, \dots, \lambda_n$ the eigenvalues of $DF(0)$ (that we assume diagonalizable). If there are $ C > 0 $ and $ v \in \N $ such that for all $l=1,\dots,n$ and $m_1, \dots, m_n \in \N$ such that $\sum m_j \geq 2$ it holds
$$ |\lambda_l - \lambda_1^{m_1} \cdots \lambda_n^{m_n} | \geq \frac{C}{(\sum_{j=1}^n m_j)^v}, $$
then $F$ is holomorphically linearizable.
\end{theorem}

\begin{remark}\label{linpertu}
   The condition on the eigenvalues is not generic, however it is of full measure, and so, dense. Hence, given any automorphism $H \in Aut(\C^2)$ and a fixed saddle point $p$, we can perturb $H$ to $A\circ H$, where $A$ is an affine transformation on $\C^2$ fixing $p$ close to the identity, to make the eigenvalues of $D(A\circ H)(p)$ satisfy the linearization conditions of the theorem above. 
\end{remark}

The holomorphic linearization allows us to have good control on how the inclinations of the stable and unstable foliations vary near $W^u(p)$ and $W^s(p)$. 

For this purpose, let $H$ be an automorphism of $\C^2$ with a fixed saddle point $p$ such that the eigenvalues of $DH(p)$ satisfy the condition stated in Siegel's Theorem. Let $p'$ be a point in $W^s(p)$ such that $T_{p'}W^{s}(p)$ is transversal to the $w$-direction.

Further, let $V$ be a neighborhood of $p'$ in which the stable foliation $\F^s$ is defined and $I^s: V \to \C $ be a $C^{1 + \alpha}$ function such that, for every $(z,w) \in V$, the value of $I^s(z,w)$ is the inclination of $T_{(z,w)}\F^s$ with respect to the $z$-direction, that is
\[
T_{(z,w)}\F^s = \{(v, I^s(z,w) \cdot v),\; v \in \C\}.
\]
We can consider that $I^s: V \subset \R^4 \to \R^2$, again because of remark. \ref{holocur}

\begin{lemma}\label{linearizableH}
In the context above, for every point $(z,w)$ belonging to $W^s(p)$, both $\frac{\partial I^s}{\partial z}(z,w)$ and $\frac{\partial I^s}{\partial w}(z,w)$ are $2 \times 2$ conformal matrices.
\end{lemma}

\begin{proof}
Let $\Psi $ be a holomorphic linearization of $H$. By iteration, the domain of $\Psi$ may be extended to some $V' \subset \C^2$ such that:
\begin{itemize}
    \item $V' = R \cdot \D \times \delta\cdot  \D$.
    \item $\Psi(V')$ contains a small neighborhood of $p'$.
\end{itemize}

Moreover, we can consider that the foliation $\F^s$ has been constructed in the region $V$ by iteration from a small neighborhood of $p$. That way, we can consider $\F^s$ to be defined in all of $\Psi(V')$ as well. Let $ \hat{\F}^s \coloneqq \Psi^{-1}(\F^s) $ and $\hat{I}^s(v^s,v^u)$ denote the inclination of $\hat{\F}^s$ in the point $(v^s,v^u)$, that is, 
\[
T_{(v^s, v^u)}\hat{\F}^s = \left\{(v, \hat{I}^s(v^s,v^u) \cdot v), \; v \in \C \right\}.
\]
Since $\Psi^{-1}(W^s(p)  ) = R \cdot \D \times \{0\}$, we have $\hat{I}^s(v^s, 0) = 0$ for all $v^s \in R \cdot \D$. So, $\frac{\partial \hat{I}^s}{\partial v^s}(v^s,0) = 0$. We now show that $\frac{\partial \hat{I}^s}{\partial v^u}(v^s,0) = 0$ for all $v^s \in R \cdot \D$ as well. 

Let $\hat{H} \coloneqq \Psi^{-1} \circ H \circ \Psi $. Notice that $\hat{\F}$ is semi-invariant by $\hat{H}$. Since $\hat{H}(v^s,v^u) = (\lambda^s \cdot v^s, \lambda^u \cdot v^u)$,
\[
T_{\hat{H}(v^s, v^u)}\hat{\F}^s = \left\{(\lambda^s \cdot v, \lambda^u \cdot \hat{I}^s(u,v) \cdot v), \; v \in \C \right\}.
\]
Therefore,
\begin{equation}\label{inclination}
    \hat{I}^s\left((\hat{H}(v^s,v^u)\right) =(\lambda^s)^{-1} \cdot \lambda^u \cdot \hat{I}^s(v^s,v^u). 
\end{equation}

Keep $v^s$ fixed now. If $\delta \cdot (\lambda^u)^{-m-1}  \le \lv v^u_m \rv < \delta \cdot (\lambda^u)^{-m} $, then all iterates 
\[
\hat{H}(v^s,v^u_m), \dots, \hat{H}^m(v^s,v^u_m)
\]
belong to the region $V'$. Meanwhile, there are $C>0 $ and $\delta' > 0$ such that $\lv \hat{I}^s \rv < C$ on $\delta' \cdot \D \times \delta \cdot \D$. If $m$ is sufficiently large, then $\hat{H}^m(v^s,v^u_m)$ belongs to this region, and so 
\[
\lv \hat{I}^s(\hat{H}^m(v^s,v^u_m)) \rv < C
\]
By equation \ref{inclination}, this implies
\[
\lv \hat{I}^s(v^s,v^u_m) \rv < R\cdot (\lambda^s)^{m} \cdot (\lambda^u)^{-m},
\]
hence 
\[
\frac{\lv \hat{I}^s(v^s,v^u) \rv}{\lv v^u_m \rv  } <  (\lambda^s)^m \cdot (\lambda_u)^{-1} 
\]
and making $m \rightarrow  0$ we conclude that
\[
\frac{\partial \hat{I}^s}{\partial v^u}(v^s,0) = \lim_{v^u \rightarrow 0} \frac{\lv \hat{I}^s(v^s,v^u) \rv}{\lv v^u \rv } = 0. 
\]
Returning to the original setting, if $(z,w) = \Psi(v^s,v^u)$ is a point in $V$, then 
\[
T_{(z,w)}\F^s = \left\{\left(\left(
\frac{\partial \Psi_1}{\partial v^s} + \frac{\partial \Psi_1}{\partial w} \cdot \hat{I}^s(v^s,v^u) \right)\cdot v
, \left(\frac{\partial \Psi_2}{\partial v^s} + \frac{\partial \Psi_2}{\partial w}\cdot \hat{I}^s (v^s,v^u) \right)\cdot v\right) ,\; v \in \C \right\},
\]
where $\Psi_j = \Pi_j \circ \Psi$ for $j=1,2$ and all the partial derivatives are calculated at $(v^s,v^u)$. Since $\Psi$ is holomorphic, we can think of them as complex numbers involved in an expression with the complex number $\hat{I}^s(v^s, v^u)$.

By the transversality hypothesis on $W^s(p')$, the partial derivative $\frac{\partial \Psi_1}{\partial v^s}$  in invertible for $(v^s,v^u)$ close to $\Psi^{-1}(p')$. For points $(z,w)$ sufficiently close to $W^s(p')$, the corresponding value of $v^u$ will be close enough to $0$ so that $\hat{I}(v^s,v^u)$ is close to zero and 
\[
\frac{\partial \Psi_1}{\partial v^s}(v^s,v^u) + \frac{\partial \Psi_1}{\partial v^u}(v^s,v^u) \cdot \hat{I}^s(v^s,v^u)
\]
is invertible. Therefore, in a thin strip around a small part of $W^s(p)$ containing $p'$, we have 
\[
I^s = \left(\frac{\partial \Psi_2}{\partial v^s} + \frac{\partial \Psi_2}{\partial v^u}\cdot \hat{I}^s \right) \cdot\left( \frac{\partial \Psi_1}{\partial v^s} + \frac{\partial \Psi_1}{\partial v^u}\cdot \hat{I}^s\right)^{-1} \circ \Psi^{-1}. 
\]
Using the chain rule, the conclusion follows from the fact that 
\[
\frac{\partial \hat{I}^s}{\partial v^s}(v^s,0) = \frac{\partial \hat{I}^s}{\partial v^u}(v^s,0) = 0,
\]
and the fact that $\Psi$ is holomorphic.

\end{proof}

The analogous statement is also true for the unstable foliation and its inclination function $I^u$. Moreover, the map $H$ does not need to be holomorphically linearizable. 

\begin{lemma}\label{generalH}
    The same conclusion of lemma \ref{linearizableH} is true when the condition on the eigenvalues of $DH(p)$ is removed. 
\end{lemma}

\begin{proof}
      For this, let $\ti{H}$ be any automorphism close to $H$. Let $\ti{p}$ be the continuation of the fixed saddle point $p$ and consider the stable foliation $\ti{\F}^s$ defined near the point $\ti{p}$. Also, let  $\ti{I}^s(z,w)$ be the function measuring the inclinations of the tangent spaces $T_{(z,w)}\ti{\F}^s$.
      
    The function $\ti{I}^s$ converges to $I^s$ in the $C^{1+\al}$ topology when $\ti{H} \rightarrow H$, as pointed out in remark \ref{tangsmooth}. Hence, the result follows by choosing the maps $\ti{H}$ to be linearizable, as explained in remark \ref{linpertu}.
\end{proof}

\begin{lemma}\label{linsub}
The tangent space of the disk of tangencies $D_0$ at $ q$ is a complex linear subspace of $\C^2$. 
\end{lemma}

\begin{proof}
The function $I^s(w,z)$ used in lemmas \ref{linearizableH} and \ref{generalH} can be represented in terms of $\vp_0 $ as 
\[
I^s(z,w) = \left( \frac{\partial {\vp}^s_{0,0}}{\partial z} \circ ({\vp}^s_0)^{-1}\right) (z,w).
\]
Analogously,
\[
I^u(z,w) = \left( \frac{\partial {\vp}^u_{0,0}}{\partial z} \circ ({\vp}^u_0)^{-1}\right) (z,w).
\]
Notice that the function $f_0(z,w)$ defined in the beginning of proof of lemma \ref{disktan} then satisfies
\[
f_0 = I^u(z,w) - I^s(z,w).
\]
Hence, by lemmas \ref{linearizableH} and \ref{generalH} and the fact that $q$ belongs to the intersection of $W^u_0(p_0)$ and $W^s_0(p_0)$,
\[
\frac{\po f_0}{\po z}(q_1,q_2) \qquad \text{and} \qquad \frac{\po f_0}{\po w}(q_1,q_2)
\]
are both conformal matrices (this was already known for the first matrix by the end of the proof of lemma \ref{disktan}). 

However, by the implicit function theorem, the tangent space of $D_0$ is the space of vectors 
    \[
T_q D_0 =\left\{\left(v,\frac{\partial f_0}{\partial w}(q_1,q_2) \left[\frac{\partial f_0}{\partial z} (q_1,q_2) \right]^{-1} v\right), \; v \in \R^2\right\},
    \]
 and since the matrices above are conformal, $T_q D_0$ is a complex linear subspace of $\C^2$.
\end{proof}

\begin{corollary}\label{confconfig}
The derivatives of the configurations $h_0^u$ and $h_0^s$ at $\mu = 0$,
\[Dh_0^u(0): \C \to \C \qquad \text{and} \qquad Dh_0^s(0) : \C \to \C, \] are conformal transformations, that is, they correspond to multiplication by a complex constant.
\end{corollary}

\begin{proof}
    Considering the definitions of $h^u_\mu$ and $h^s_\mu$ in lemma \ref{evident}, the conclusion follows from Lemma \ref{linsub} and Lemma \ref{projections} together.
\end{proof}

 \section{Proof of the Theorem \ref{main}}\label{sec4}

 We are going to need the following result, which may be known by some readers.

     \begin{lemma}\label{genericv} Let $ z=R_a\cdot e^{\mi a}$ and  $w=R_b\cdot e^{\mi b};$ be complex numbers with $ a,b,R_a,R_b \in \R$. If
 \[
a\log R_b- b \log R_a, \qquad \pi \log R_a,  \qquad \pi \log R_b
 \]
 are linearly independent over $\Q$, then the set $X=\{u=z^m\cdot w^n\in \C;\, m,n \in \Z\}$ is dense on $\C$. 
 
 Additionally, for each $v \in \C$, $\delta>0$ and $M=M_{v,\,\delta} $ sufficiently large, there is a sequence of integers $(m_i,n_i)_{i \in \N}$ such that, 
 
 \begin{enumerate}[(i)]
     \item $z^{m_i}\cdot w^{n_i} $ belongs to $B_\delta(v)$ for all $i \in \N$.
     \item $0 < m_{i + 1} - m_i < M$ and $|n_{i+1}- n_i| < M$.
 \end{enumerate} 

\end{lemma} 

\begin{proof} Applying logarithms, we need to show that 
\[
X'= \{(m\cdot \log R_a + n \cdot \log R_b, ma + nb+2k\cdot \pi) \in \R^2; m,n,k \in \Z \}
\]
is dense in $\R^2$. By Kronecker's Theorem, we know that if $e_1=(1,0)$, $e_2=(0,1)$ and $v=(c,d)$ are vectors such that $1,\,c,\,d$ are linearly independent over $\Q$, then the set
\[
Y 
=\{ m\cdot e_1 + n \cdot e_2 + k \cdot v; \, m,\,n,\,k \in \Z \}
\]
is dense on $\R^2$. Thus, consider the linear operator 
\[T=\begin{bmatrix}
    \log R_a     &  \log R_b \\ a & b
\end{bmatrix}
\]

and constants 
\[
c=\frac{2\pi \log R_b}{a\log R_b - b \log R_a}\quad \text{and} \quad d= - \frac{2\pi \log R_a}{a\log R_b - b \log R_a}.
\] 
Under these choices, 
\[
(m\cdot \log R_a + n \cdot \log R_b, ma+nb+2k\cdot \pi) = T( m\cdot e_1 + n \cdot e_2 + k \cdot v),
\] 
so $X'=T(Y)$ is dense if $a\log R_b- b \log R_a$, $\pi \log R_a$ and  $\pi \log R_b$ are linearly independent over $\Q$.
 
Suppose now that we are under this hypothesis and $z^{m_i}\cdot w^{n_i} \in B_{\delta}(v)$ for a fixed $v$. There are $\ve>0 $ and $v'$ close to $1$ such that, for every $ u \in B_{\ve}(v')$, the point $u\cdot z^{m_i}\cdot w^{n_i} $ belongs to $ B_{\delta}(v)$ ($v'$ depends on the value of $z^{m_i}\cdot w^{n_i} \in B_{\delta}(v)$).

Since $X$ is dense in $\R^2$, there are $0 < N < 3N <  M/2$ and $s_1,t_1,s_2,t_2$ such that:

\begin{itemize}\setlength\itemsep{.5em}
    \item $|s_1| < N$ and $|t_1|< N$;
    \item $ u_1 = z^{s_1}w^{t_1}$ satisfies $|u_1 - v'| < \frac{\ve}{4}$;
    \item $3N < |s_2| < M/2$ and $ 3N < |t_2|< M/2$;
    \item $ u_2 = z^{s_2}w^{t_2}$ satisfies $|u_2 - 1| < \frac{\ve}{4|v'|}$;
\end{itemize}

and thus, if $\ve$ is sufficiently small, $u_1u_2 $ and $u_1u_2^{-1}$ belong to $B_{\ve}(v')$ as desired. If $s_2 > 0$, then 
\[
u_1u_2 = z^{s_1 + s_2} \cdot w^{t_1 + t_2} = z^{m'}\cdot w^{n'}
\]
with $0 < m' <  N$ and $|n'| < M$ and so we can choose $m_{i+1} = m_i + m'$ and $n_{i+1} = n_i + n'$. If $s_2 < 0$, then 
\[
u_1u_2^{-1} = z^{s_1 - s_2} \cdot w^{t_1 - t_2} = z^{m'}\cdot w^{n'}
\]
also with $0 < m' <  N$ and $|n'| < M$ and so we can choose $m_{i+1} = m_i + m'$ and $n_{i+1} = n_i + n'$ as well.

The sequence is constructed repeating this process.

\end{proof}

Notice that the conditions of lemma \ref{genericv} on the complex numbers $z$ and $w$ are generic. Conclusions $(i)$ and $(ii)$ also hold for many more complex numbers $z$ and $w$ using $v = 1$ instead.

\begin{corollary}\label{coro}
    Let $z = R_ae^{ia}$ and $w = re^{ib}$ be complex numbers such that $\lv z\rv \neq 1$ and $\lv w \rv \neq 1$. Then, for each $\delta>0$ and $M=M_{\delta} $ sufficiently large, there is a sequence of integers $(m_i,n_i)_{i \in \N}$ such that, 
 
 \begin{enumerate}[(i)]
     \item $z^{m_i}\cdot w^{n_i} $ belongs to $B_\delta(1)$ for all $i$.
     \item $0 < m_{i + 1} - m_i < M$ and $|n_{i+1}- n_i| < M$.
 \end{enumerate} 
\end{corollary}

\begin{proof}
We remember that, by Kronecker's Theorem, given any $\eta \in \R$ and $\ve > 0$, there exist integers $p$ and $q$ such that $ \lv q - p\eta \rv < \ve $. Proceeding similarly as in the proof of Lemma \ref{genericv}, we can conclude that given any pair of real numbers $\eta_1 \neq 0$ and $\eta_2 \neq 0$, there exists $M_\ve>0$ and a sequence $(m_i,n_i)_{i \in \N}$ such that $0 < m_{i + 1} - m_i < M_\ve $, $\lv n_{i+1} - n_i \rv < M_\ve$ and $ \lv m_i\eta_1 + n_i\eta_2 \rv < \ve $.

Going back to the statement of our lemma, we divide into two cases.

I) $ a\log R_b - b \log R_a = 0$.

Since $\log R_b \neq 0$ and $\log R_a \neq 0$, there exists $M_\ve>0$ and a sequence $(m_i,n_i)$ such that $0 < m_{i + 1} - m_i < M_\ve $, $\lv n_{i+1} - n_i \rv < M_\ve$ and $ \lv m_i\log R_a + n_i \log R_b \rv < \ve\cdot \min{\{1,\log R_b\}}$
Then 
\[
\lv m_i a + n_i b \rv = \lv \frac{b}{\log R_b} (m_i\log R_a + n_i \log R_b ) \rv < \lv b \rv \cdot \ve
\]
and, shrinking $\ve$ and enlarging $M_\ve$ if necessarily, we get the desired sequence.

II) $a\log R_b - b \log R_a \neq 0$.

Proceeding as in the proof of lemma \ref{genericv}, it is sufficient to prove that if $1$, $c \neq 0$ and $d \neq 0$ are real numbers, then, for any $\ve>0$, there are $M_\ve > 0$ and a sequence of integers $(m_i,n_i,k_i)$ such that $0 < m_{i + 1} - m_i < M_\ve $, $\lv n_{i+1} - n_i \rv < M_\ve$, $ \lv m_i + k_i c \rv < \ve $ and $\lv n_i + k_i d \rv < \ve $. Because of lemma \ref{genericv}, we need only to consider the case in which $1$, $c$ and $d$ are rationally dependent. Therefore, there are integers $\gamma_1$, $\gamma_2$ and $\gamma_3$, not all zero, such that
\[
\gamma_1 + \gamma_2 c  + \gamma_3 d = 0.
\]
We can assume without loss of generality that $\gamma_3 \neq 0$. Given $\ve > 0$, let $M_\ve$ and $(\ti{m}_i,\ti{k}_i)$ be such that $0 < \ti{m}_{i + 1} - \ti{m}_i < M_\ve $, $\lv \ti{k}_{i+1} - \ti{k}_i \rv < M_\ve$ and $ \lv \ti{m}_i + \ti{k}_i c \rv < \ve$. 

Since $d = \frac{-\gamma_1 - \gamma_2 c}{\gamma_3}$, making $m_i = \gamma_3 \ti{m}_i$, $n_i =  \gamma_1 \ti{k}_i - \gamma_2 \ti{m}_i$ and $k_i = \gamma_3 \ti{k}_i$, we conclude that
\begin{align*}
    \lv m_i + k_i c \rv   & < \gamma_3 \ve\\
    \lv n_i + k_i d \rv  = \lv - \gamma_2 \ti{m}_i - \gamma_2 \ti{k}_i c\rv  &<  \gamma_2 \ve.
\end{align*}

Shrinking $\ve$ and enlarging $M_\ve$, we get a sequence $(m_i,n_i)$ as desired.

\end{proof}

\begin{proof}(Of theorem \ref{main})
    Because of lemma \ref{evident}, we need to show that the set of values $\mu$ for which $h^u_\mu(K^u_\mu)$ and $ h^s_\mu(K^s_\mu)$ have stable intersections has positive inferior density at $\mu = 0$.

    Let $g_\mu^u$ and $g_\mu^s$ be the maps defining the Cantor sets. Following the notation of subsection \ref{confcant} and Definition \ref{linv}, let $\ute^u = (\dots, a^u, a^u, a^u)$ and $\ute^s = (\dots, a^s, a^s, a^s)$ be the infinite words such that 
    \[
    f^u_{\mu, (a^u,a^u)}(0) = 0 \qquad \text{and} \qquad f^s_{\mu,( a^s,a^s)}(0) = 0.
 \]
Notice they do not depend on $\mu$. 

The idea is to show that for many values of $\mu$ and suitable integer values of $m > 0$ and $n$ there are stable intersections between 
\[
h^u_\mu(K^u_\mu \cap G(\ute^u_{m}))  \qquad \text{and} \qquad h^s_\mu(K^s_\mu \cap G(\ute^s_{n})).
\]
To simplify the notation, we will write 
\begin{align*}
    f^u_m \coloneqq f^u_{\ute^u_{m}} & \qquad \text{and} \qquad f^s_n \coloneqq f^s_{\ute^s_{n}}.
\end{align*}

Because of the invariance of the Cantor sets, we can rewrite
\begin{align*}
    h^u_\mu(K^u_\mu \cap G(\ute^u_m)) & = (h^u_\mu \circ f^u_{m}) (K^u_\mu \cap G(a^u)), \\
    h^s_\mu(K^s_\mu \cap G(\ute^s_n)) & = (h^s_\mu \circ f^s_{n}) (K^s_\mu \cap G(a^s)) .
\end{align*}

We go further and denote $h^u_{\mu,m} \coloneqq h^u_\mu \circ f^s_{n} $ and $h^s_{\mu,n} \coloneqq h^s_\mu \circ f^s_{n} $.

The hypothesis in the statement of Theorem \ref{main} is that $(k^{\ute^u_0},\zeta \cdot k^{\ute^s_0} + \nu)$ is a pair of configurations of the Cantor sets $K^u_0$ and $K^s_0$ that has stable intersections. But these are configurations of the pieces $G(a^u)$ and $G(a^s)$ respectively. As observed below definition \ref{stableint}, the property of stable intersection is preserved by composition of the configuration pair on the left by any $A \in Aff(\R^2)$. Seeing that for small values of $\mu$ the pairs $K_0^u$ and $K_\mu^u$ are close to each other, and also are $K_0^s$ and $K_\mu^s$, the problem is reduced to showing that the configuration pairs
\[
(A_{h^u_{\mu,m}} \circ h^u_{\mu,m}, A_{h^u_{\mu,m}} \circ h^s_{\mu,n}) \qquad \text{and} \qquad (k^{\ute^u_0},\zeta \cdot k^{\ute^s_0} + \nu)
\]
are sufficiently close to each other.

More precisely, by the continuous dependence of the Cantor sets on $\mu,$ if $\mu$ is sufficiently small, there is $\ve > 0 $ such that for every pair $(h,h')$ $\ve$-close to $(k_\mu^{\ul{\theta}^u},\zeta\cdot  k_\mu^{\ul{\theta}^s} + \nu)$, there are stable intersections between $h({K}_\mu^u)$ and $h'({K}_\mu^s)$.

Let  $ \Delta(\mu) \coloneqq h_\mu^{s}(0) - h_\mu^{u}(0)$. Condition (\ref{itm:IV}) on the family $H_\mu$ implies that 
\[
\frac{\partial \Delta}{\partial \mu} (0) = B \in GL(\R^2).
\] 

Remember that because of corollary \ref{confconfig}, the product of matrices\[\left(Dh^u_0(0)\right)^{-1} \cdot Dh^s_0(0)\] is conformal and so can be identified with a complex number. We need to prove the lemma below.

    \begin{lemma}\label{1} Let $\lambda^u_\mu$ and $\lambda^s_\mu$ be the eigenvalues of $DH_\mu$ at $p_\mu$. 
    There is $\delta > 0$ such that, if $m>0$ is sufficiently large, and $n$ satisfy 
    \begin{equation} \label{h1}
        \lv (\lambda^u_0)^m \cdot (\lambda^s_0)^{n} \cdot  \left(Dh^u_0(0)\right)^{-1} \cdot Dh^s_0(0) - \zeta\rv < \delta
    \end{equation}
and $\mu$ (sufficiently small) satisfies
\begin{equation}\label{h2}
    \lv\,Dh^u_0(0 )^{-1}\cdot B \cdot (\lambda^u_0)^m   \cdot \mu - \nu \rv < \delta ,
\end{equation}
then 
\[
(A_{h^u_{\mu,m}} \circ h^u_{\mu,m}, A_{h^u_{\mu,m}} \circ h^s_{\mu,n}) \qquad \text{and} \qquad (k^{\ute^u_0},\zeta \cdot k^{\ute^s_0} + \nu)
\]
are $\ve$-close to each other and so there are stable intersections between $h^u_\mu({K}_\mu^u)$ and $h^s_\mu({K}_\mu^s)$.
    \end{lemma}

\begin{proof}

Following the notation of lemma \ref{scale}, let
\[
\mathfrak{h}^u_{\mu,m} \coloneqq A_{h^u_{\mu,m}} \circ h^u_{\mu,m} \circ \left(k^{\ute^u \ute^u_m}\right)^{-1} \qquad \text{and} \qquad \mathfrak{h}^s_{\mu,n} \coloneqq A_{h^s_{\mu,n}} \circ h^s_{\mu,s} \circ \left(k^{\ute^s \ute^s_n}\right)^{-1}.
\]
Notice that $\ute^u \ute^u_m = \ute^u$ and $\ute^s \ute^s_n = \ute^s$. Then,
\[
(A_{h^u_{\mu,m}} \circ h^u_{\mu,m}, A_{h^u_{\mu,m}} \circ h^s_{\mu,n}) = (\mathfrak{h}^u_{\mu,m}  \circ k^{\ute^u}, A_{h_{\mu,m}^u} \circ A_{h_{\mu,n}^s}^{-1}\circ \mathfrak{h}^s_{\mu,n}  \circ k^{\ute^s})
\]

By lemma \ref{scale}, if $m$ and $n$ are very large, then $\mathfrak{h}^u_{\mu,m}$ and $\mathfrak{h}^s_{\mu,n}$ are very close to the identity. Therefore, it suffices to show that, given any $\ve >0$, there is some $\delta > 0$ such that if equations (\ref{h1}) and (\ref{h2}) are satisfied, then 
\[
A_{h_{\mu,m}^u} \circ A_{h_{\mu,n}^s}^{-1} \qquad \text{and} \qquad \zeta \cdot z + \nu
\]
are $\ve$-close on some large compact disk $R \cdot \D$ contained in $\C$. For this, all one needs is
\begin{equation}\label{i1}
    \lv D \left( A_{h_{\mu,m}^u} \circ A_{h_{\mu,n}^s}^{-1} \right) - \zeta \rv < \frac{\ve}{2R} \qquad \text{and} \qquad  \lv  A_{h_{\mu,m}^u} \circ A_{h_{\mu,n}^s}(0) - \nu \rv < \frac{\ve}{2}.
\end{equation}

First, the map $ A_{h_{\mu,m}^u} \circ A_{h_{\mu,n}^s}^{-1} $ is not necessarily in $Aff(\C)$. However, since the maps $g^u_\mu$ and $g^s_\mu$ are constructed from the dynamics of the automorphism $H_\mu$ (see remark \ref{eigenv}), one has 
\[
Df_{\mu,(a^u, a^u)}(0) = (\lambda^u_\mu)^{-1}\qquad \text{and} \qquad Df_{\mu,(a^s, a^s)}(0)= \lambda^s_\mu.
\]
Considering the definition of $A_h$ for any configuration $h$, we calculate
\begin{align*}
    A_{h_{\mu,m}^u}(z) & = (\lambda^u_\mu)^{m} \left(Dh^u_\mu(0)\right)^{-1} \cdot (z  -h^u_\mu(0)),  \\
    A_{h_{\mu,n}^s}(z) & = (\lambda^s_\mu)^{-n} \left(Dh^s_\mu(0)\right)^{-1} \cdot  (z  -h^s_\mu(0)).
\end{align*}
Hence
\begin{equation*}
\begin{split}
       A_{h_{\mu,m}^u} \circ A_{h_{\mu,n}^s}^{-1} =\;  & (\lambda^u_\mu)^{m} \cdot \left(Dh^u_\mu(0)\right)^{-1} \cdot Dh^s_\mu(0)\cdot (\lambda^s_\mu)^{n} \cdot z\\ &  +  (\lambda^u_\mu)^{m} \cdot \left(Dh^u_\mu(0)\right)^{-1}\cdot  \left[ h^s_\mu(0) - h^u_\mu (0)\right].
\end{split}
\end{equation*}

We now prove that for $\delta $ and $\mu$ sufficiently small 
\begin{equation}\label{i2}
        \lv (\lambda^u_\mu)^{m} \cdot \left(Dh^u_\mu(0)\right)^{-1} \cdot Dh^s_\mu(0)\cdot (\lambda^s_\mu)^{n} - (\lambda^u_0)^m \cdot (\lambda^s_0)^{n} \cdot \left(Dh^u_0(0)\right)^{-1} \cdot Dh^s_0(0) \rv  < \frac{\ve}{4R}
\end{equation}
and 
\begin{equation}\label{i3}
   \lv  (\lambda^u_\mu)^{m} \cdot \left(Dh^u_\mu(0)\right)^{-1} \left[ h^s_\mu(0) - h^u_\mu (0)\right] - Dh^u_0(0 )^{-1}\cdot B \cdot (\lambda^u_0)^m   \cdot \mu  \rv < \frac{\ve}{4},
\end{equation}
and so the inequalities in (\ref{i1}) are satisfied.

If condition (\ref{h2}) is satisfied, then there is some constant $C > 0$ such that $ \lv \mu \rv < C (\lambda^u_0)^{-m} $. Because the eigenvalues depend $C^{r-1}$ on $\mu$, there is some constant $C' > 0$ such that 
\[
\lv \lambda^u_\mu - \lambda^u_0 \rv < C' \lv \mu \rv^{(r-1)} \qquad \text{and} \qquad   \lv \lambda^s_\mu - \lambda^s_0 \rv < C' \lv \mu \rv^{(r-1)}.
\]
Hence, 
\[
\left(1- CC' (\lambda_0^u)^{1-m(r-1)} \right)^m < \lv \frac{(\lambda^u_\mu)^m } {(\lambda^u_0)^m}  \rv < \left(1+ CC' (\lambda_0^u)^{1-m(r-1)} \right)^m.
\]
But the lower and upper bounds converge to $1$ as $m \rightarrow +\infty$. Analogously, we can conclude that if $n$ is sufficiently large, then 
\[
\frac{(\lambda^s_\mu)^n } {(\lambda^s_0)^n}
\]
is also very close to $1$. 

At the same time, if $m$ is sufficiently large, then $\mu$ is very small and so
\[
\left(Dh^u_\mu(0)\right)^{-1} \cdot Dh^s_\mu(0) \qquad \text{and} \qquad  \left(Dh^u_0(0)\right)^{-1} \cdot Dh^s_0(0)
\]
are very close to each other. We conclude that if $m$ and $n$ are sufficiently large, inequality (\ref{i2}) is satisfied. 

Meanwhile, $h^s_\mu(0) - h^u_\mu (0) = \Delta(\mu)$ and so
\[
\lv h^s_\mu(0) - h^u_\mu (0) - B \mu \rv < C'' \lv\mu \rv^r
\]
for some constant $C'' > 0$. Proceeding as above, we conclude that inequality (\ref{i3}) is valid as well when $m$ is sufficiently large. This concludes the proof.
\end{proof}

By corollary \ref{confconfig}, we can set
    \[
    \zeta \coloneqq \left(Dh^u_0(0)\right)^{-1} \cdot Dh^s_0(0)\in \C.
    \]
If we make $z = \lambda^u_0$ and $w = \lambda^s_0$, Corollary \ref{coro} gives us a large $M$ and a sequence of integers $(m_i,n_i)_{i \in \N}$ with $m_i>0$ sufficiently large and $0 < m_{i+i} - m_i < M$ such that the condition (\ref{h1}) of Lemma \ref{1} is satisfied. There is a constant $C'''> 0$ such that condition (\ref{h2}) of Lemma \ref{1} is satisfied whenever $\mu$ belong to the ball of radius $r_i = C'''(\lambda^u_0)^{-m_i}$ centered at 
\[
z_i = (\lambda_0^u)^{ -m_i} \cdot B^{-1} \cdot Dh^u_0(0) \cdot \nu.
\]
Now, fix $\rho > 0$. Restricting to a subsequence and enlarging $M$ if necessary, we can assume all those balls are disjoint and 
\[
\rho (\lambda^u_0)^{-M} < \lv z_0 \rv +  r_0  < \rho.
\]  
This implies $(\lambda_0^u)^{-m_0} > \frac{\rho (\lambda^u_0)^{-M}}{C'''}$. 
Hence
    \begin{align*}
    m(C_{pers}\cap \rho \cdot \D) & \geq\sum_{i \in \N} \pi r_i^2\\
    & \geq(C''' )^2 \sum_{i \in \N} (\lambda^u_0)^{-2m_i}\\
    & \geq (C''' )^2 \sum_{i \in \N} (\lambda^u_0)^{-2m_0 - Mi} \\
    & \gtrsim (\lambda^u_0)^{-2m_0} \gtrsim \rho^2,
\end{align*}
and so $C_{pers}$ has positive inferior density at $\mu = 0$.
\end{proof}

\begin{remark}\label{final}
    From the proof of Theorem \ref{main}, we conclude that if $\lambda_0^u$ and $\lambda_0^s$ satisfy the generic condition of Lemma \ref{genericv}, then there is no need to set $\zeta \coloneqq \left(Dh^u_0(0)\right)^{-1} \cdot Dh^s_0(0) \in \C$. In this case, because of Lemma \ref{genericv}, if $(k^{\ute^u},\zeta \cdot  k^{\ute^s} + \nu)$ has stable intersections for some complex values $\zeta$ and $\nu$, then the set $C_{pers}$ has positive inferior density at $\mu = 0$.
\end{remark}

\section*{Appendix}\label{appendix}

The hypothesis in Theorem \ref{main} may appear very specific to some readers, so we provide some circumstances in which they hold or may hold. In the aforementioned work \cite{Buzz}, Buzzard constructs families of automorphisms which unfold homoclinic intersections associated to horseshoes close a ``very thick'' piecewise linear model. It was later shown in \cite{araujo_moreira_2023.1} (section 4), with the aid of the \emph{recurrent compact criterion}, that the linearized versions of the Cantor sets arising in this context have stable intersections.

Furthermore, in \cite{amz}, the authors prove that stable intersections of the linearized Cantor sets hold generically under the hypothesis that the sum of the Hausdorff dimensions of those sets is larger than $2$. However, there is a caveat: the perturbations done there are executed on the map defining the conformal Cantor sets, not on the dynamics (the automorphisms) generating then (and the horseshoe they are associated with). It is not well understood yet how perturbations on the horseshoe impact the Cantor sets. Currently, the authors of \cite{amz} are working on a version of this result for pertubations on the dyanmics. 

Now we prove that any homoclinic tangency can be perturbed to a quadratic one and be unfolded by a parameterized family as described in section \ref{sec3}.

\begin{lemma}
Let $H$ be an automorphism of $\C^2$ with a fixed point $p$ such that $W^u(p)$ and $W^s(p)$ are tangent to each other at the point $q$. Then, arbitrarily close to $H$, there are an automorphism $H_0 \in Aut(\C^2)$ and a family $\{H_\mu\}_{\mu \in \D}$ of automorphisms unfolding a generic homoclinic tangency, that is, satisfying conditions (\ref{itm:I}) - (\ref{itm:IV}) appearing in section \ref{sec3}. Further, the homoclinic tangency of $H_0$ occurs close to the point $q$, the homoclinic tangency of $H$.
\end{lemma}

\begin{proof}
    By a change of coordinates, we can suppose that $q = (0,0)$ and that $W^u(q)$ and $W^s(q)$ are both tangent to the $z$-direction at $q$.

    For small $\ve > 0$, let $V_0^u$ (resp. $V^s_0$) be the connected component of $W^u (q) \cap (D_\ve(0) \times \C)$ (resp. $W^s (q) \cap (D_\ve(0) \times \C)$) containing $q$. If $\ve$ is sufficiently small, then $V^u_0$ (resp. $V^s_0$) is the graph of a function $f^u \text{ (resp. }f^s): D_\ve(0) \to \C $. We have $(f^u)'(0) = (f^s)'(0) = 0$. We consider only the case in which $(f^u)''(0) = (f^s)''(0)$, the other one ($(f^u)''(0) \neq (f^s)''(0)$) follows from the same arguments presented below.

    For $-k^u \le k \le - 1$, let $V^u_{k}$ be small open disks contained in $W^u(p)$ such that 
    \begin{itemize}\setlength\itemsep{.5em}
        \item $V^u_k \cap W^u_\ve(p) = \emptyset$;
        \item $H(V^u_k) \supset\overline{ V^u_{k+1}}$, for $-k^u \le k \le - 1$;
        \item $ H^{-1}(\overline{V^u_{-k^u}}) \subset W^u_\ve(p)$.
    \end{itemize}

    Analogously, for $1 \le k \le k^s$, let $V^s_{k}$ be small open disks contained in $W^s(p)$ such that 
    \begin{itemize}\setlength\itemsep{.5em}
        \item $V^s_k \cap W^s_\ve(p) = \emptyset$;
        \item $H^{-1}(V^s_k) \supset\overline{ V^s_{k-1}}$, for $-k^u \le k \le - 1$;
        \item $ H(\overline{V^s_{k^s}}) \subset W^s_\ve(p)$.
    \end{itemize}
    Let 
    \[
    U = W^u_\ve(p) \cup W^s_\ve(p) \cup \bigcup_{k\neq 0} V^u_k \cup \bigcup_{k\neq 0} V^s_k. 
    \]
    For any $\alpha \in \C$, let $\pi_\alpha$ be the projection of $\C^2$ onto the complex line $l_{\alpha} = \{(z,\alpha \cdot z), \; z \in \C\}$. By shrinking $\ve$ if necessary and adjusting the open disks $V ^u_{k}$ and $V^s_{k}$, we can find some $\alpha \in \C$ such that $\pi_\al(\overline{V^u_0})$ and $\pi_\al(\overline{V^s_0})$  are not intersected by $\pi_\al(\overline{U})$.

    Therefore, making a linear change of coordinates, we can consider that $\Pi_1(\overline{V^u_0}) $ and $\Pi_1(\overline{V^s_0})$ are not intersected by $\Pi_1(\overline{U})$, where $\Pi_1$ is the projection onto the first coordinate. But now $V^u_0$ and $V^s_0$ are not tangent to the $z$-direction at $(0,0)$ anymore. Nonetheless, we can consider that sufficiently close to the origin, $V^u_0$ and $V^s_0$ are both given by the graphs of two holomorphic functions
    \[
    f_{u}: D_{\ve/2}(0) \to \C \qquad \text{and} \qquad  f_{s}: D_{\ve/2}(0) \to \C
    \]
    such that ${f_u}(0) = {f_s}(0) = 0$ and ${f_u}'(0) = {f_s}'(0)$. 
    
     Shrinking $\ve$ and adjusting the disks $V^u_k$ and $V^s_k$ again if necessary, by Runge's Theorem, there are complex polynomials $p(z)$ and $q(z)$ such that
   \begin{itemize}
       \item $\lV p(z) - z^2\rV_{C^2} < \ve^9$ on $D_{\ve}(0)$;
       \item $\lV q(z) - 1\rV_{C^2} < \ve^9$ on $D_{\ve}(0)$;
       \item $\lV p(z) \rV_{C^2} < \ve^9$ and $ \lV q(z) \rV_{C^2} < \ve^9$ on a neighborhood of  $\Pi_1(\overline{U})$.
   \end{itemize}
   Consider now the automorphisms $\Psi_{\delta,\mu} : \C^2 \to \C^2$ with $\delta >0$ and $\mu \in \D$ such that 
   \[
   \Psi_{\delta,\mu}(z,w) = (z,w + \delta \cdot\left(p(z) + q(z) \mu\right)). 
   \]
By making $\delta$ small, we get automorphisms close to the identity.

Now, let $H_{\delta,\mu} \coloneqq \Psi_{\delta,\mu} \circ H$ and $p_{\delta,\mu}$ denote the continuation of the fixed saddle point (of $H$) $p$. For simplicity, we will denote the unstable and stable manifolds of $p_{\delta,\mu}$ with respect to the automorphism $H_{\delta,\mu}$ by $ W^{u}(p_{\delta, \mu})$ and $ W^{s}(p_{\delta, \mu})$ respectively.

The difference between $H_{\delta,\mu}$ and $H$ on $ \Pi_1(\overline{U})\times \C$ is of order $\delta\ve^9$. This implies, in particular, that the distance between $W^u_\ve(p)$ (resp. $W^s_\ve(p)$) and $W^u_\ve(p_{\delta,\mu})$ (resp. $W^s_\ve(p_{\delta,\mu})$) is of order $\delta\ve^9$. Iterating by $H_{\delta,\mu}$ we conclude that the continuations of the disks $V^u_0$ and $V^s_0$ over $D_\ve(0)$ are given by maps
 \[
    f_{u,\delta,\mu}: D_{\ve/2}(0) \to \C \qquad \text{and} \qquad  f_{s,\delta,\mu}: D_{\ve/2}(0) \to \C
    \]
with 
\[
f_{u,\delta,\mu} \approx \delta (z^2 + \mu) \qquad \text{and} \qquad  f_{s,\delta,\mu} \approx 0,
\]
and these approximations errors are of order $\delta\ve^9$. 

We now prove that there is some small $\mu$ for which there is a point $z_\mu$ such that
\[
f_{u,\delta,\mu}(z_\mu) = f_{s,\delta,\mu}(z_\mu)  \qquad \text{and} \qquad f'_{u,\delta,\mu}(z_\mu) = f'_{s,\delta,\mu}(z_\mu).
\]

For a fixed $\delta$, let $l_\mu(z) \coloneqq f_{u,\delta,\mu}(z) - f_{s,\delta,\mu}(z) $. If $\lv \mu \rv$ is sufficiently small in relation to $\ve^2$, Rouché's theorem implies there are two roots $q^1_\mu$ and $q^2_\mu$ of $l_\mu(z) = 0$ inside $D_\ve(0)$ (it can be $q^1_\mu = q^2_\mu$). Consider now the product
\[
k(\mu)=\frac{\partial l_\mu}{\partial z}(q^1_\mu) \cdot \frac{\partial l_\mu}{\partial z}(q^2_\mu).
\]

If $\lv \mu \rv$ is of order $\ve^3$, then $\lv k(\mu) - 4\mu\rv $ is very small in relation to $\mu$. Therefore, the map
\begin{align*}
     \overline{k}: \partial (\ve^3 \cdot \D) & \to \partial \D\\
    \mu & \mapsto \frac{k(\mu)}{\lv k(\mu) \rv}
\end{align*}
has degree one and so there is ${\und{\mu}}$ such that $\overline{k}({\und{\mu}}) = 0 = k({\und{\mu}}) $ (otherwise, $\mathbb{S}^1$ would be a retraction of the unit disk $\D$). Since $k({\und{\mu}}) = 0$, then $q^1_{\und{\mu}} = q^2_{\und{\mu}}$, as we have a double root, and we can make $z_{\und{\mu}} = q^1_{\und{\mu}} = q^2_{\und{\mu}}$. 

Let $w_{\und{\mu}} \coloneqq f_{u,\delta,{\und{\mu}}}(z_{\und{\mu}}) = f_{s,\delta,{\und{\mu}}}(z_{\und{\mu}})$. The point $(z_{\und{\mu}}, w_{\und{\mu}})$ corresponds to a tangency between $W^{u}(p_{\delta, {\und{\mu}}})$ and $ W^{s}(p_{\delta, {\und{\mu}}})$. It is quadratic because, by our approximations, we  have $f''_{u,\delta,{\und{\mu}}}(z_{\und{\mu}}) \approx 2\delta \neq 0 \approx f''_{s,\delta,{\und{\mu}}}(z_{\und{\mu}}) $. By changing the parametrization of the family by $\mu \mapsto \mu - \und{\mu}$, the family $H_{\delta,\mu}$ satisfy conditions \ref{itm:III} and \ref{itm:IV} stated in section \ref{sec3} as desired.

\end{proof}

\bibliographystyle{abbrv}
\bibliography{refs.bib}

\end{document}